\documentclass[11pt,reqno]{amsart}
\usepackage{fullpage}
\usepackage{amsmath,amsthm,verbatim,amssymb,amsfonts,amscd, graphicx,bbm}
\usepackage{graphics}
\usepackage{mathtools}
\usepackage{mathrsfs}
\usepackage{esint}
\usepackage{xcolor}
\usepackage[pdfstartview=FitH, bookmarksnumbered=true,bookmarksopen=true, colorlinks=true, citecolor=blue, linkcolor=blue,urlcolor=blue]{hyperref}%pdfborder=001, 

\usepackage{varioref} % More descriptive referencing

\newtheorem{thm}{Theorem}[section]
\newtheorem{cor}[thm]{Corollary}
\newtheorem{lem}[thm]{Lemma}
\newtheorem{prop}[thm]{Proposition}

\theoremstyle{definition}

\theoremstyle{remark}
\newtheorem{rem}{Remark}[section]

\numberwithin{equation}{section}

\def\eqdistr{\overset{\text d}{=}}
\def\BUC{\mathrm{BUC}}
\def\bP{\mathbf{P}}

\def\upd{\mathrm{d}}

\DeclareMathSymbol{\C}{\mathalpha}{AMSb}{"43}
\newcommand{\eps}{\epsilon}

\newcommand{\scrL}{\mathscr{L}}

\newcommand{\R}{{\mathbb{R}}}

\newcommand{\myi}{\mathrm{i}}
\newcommand{\dd}{\mathrm{d}}

\newcommand{\Dom}{\textsf{Dom}}
\newcommand{\ol}{\overline}

\newcommand{\bsub}{\begin{subequations}}
\newcommand{\esub}{\end{subequations}$\!$}

\newcommand{\scrF}{\mathscr{F}}

\begin{document}

\title{Stochastic Homogenization of Non-local Hamilton-Jacobi-Bellman equations}

%\author{Wenjia Jing, Qi Zhang\thanks{BIMSA,  Email: \texttt{qzhang@bimsa.com}. }}
\author[W. Jing]{Wenjia Jing}
\address[W. Jing]{Yau Mathematical Sciences Center, Tsinghua University, Beijing 100084 and Beijing Institute of Mathematical Sciences and Applications, Beijing 101408, P.R. China}
\email{wjjing@tsinghua.edu.cn}

\author[Q. Zhang]{Qi Zhang}
\address[Q. Zhang]{Beijing Institute of Mathematical Sciences and Applications, Beijing 101408, P.R. China}
\email{qzhang@bimsa.cn}
%\date{\date}

%\smallbreak \maketitle

%\vskip 0.2truein

\begin{abstract}
        In this paper we apply the method of Kosygina, Rezakhanlou and Varadhan (CPAM 2006) to establish stochastic homogenization of Hamilton-Jacobi-Bellman (HJB) equation with a vanishing non-local integro-differential operator and a convex super-linearly growing Hamiltonian in stationary ergodic random medium. Their method, first designed for homogenization of HJB equations with vanishing Laplacian operator, relies on stochastic optimal control representation of the solution, a key construction of approximate super-correctors and the technique of linking diffusion process to abstract diffusion in random media. We show how the procedures can be carried out for HJB equations with jump-diffusion. In particular, for the construction of approximate super-correctors, we represent the non-local integro-differential operator as the divergence of a regular integral operator acting on the gradient. 

\bigskip
        

        \bigskip
        
        \noindent{\bf Mathematics subject classification (MSC 2020)}: 
35B27, 47G20, 49L20.

\end{abstract}

\maketitle

%%%%%%%%%%
\section{Introduction}

We study the homogenization of the following non-local Hamilton-Jacobi-Bellman (HJB) equation with a Hamiltonian that has rapid random oscillations in space:
\begin{equation}\label{HJB}
     \left\{
   \begin{aligned}
   & \partial_t u_{\epsilon}(t,x) = \epsilon^{\alpha-1}  \mathscr{L}u_{\epsilon}(t,x) + H\left(\nabla u_{\epsilon}(t,x), \frac{x}{\epsilon}, \omega\right),\quad & (t,x) \in [0,\infty)\times\mathbb{R}^n, \\
   &  u_{\epsilon}(0,x) = u_{0}(x), \quad & x\in \mathbb{R}^n.
   \end{aligned}
   \right.
\end{equation}
Here, $\omega$ denotes an element of the underlying probability space $(\Omega, \mathscr{F}, \bP)$ on which the Hamiltonian $H$ is defined, and $\mathscr{L}$ is an integro-differential operator defined by, for $\alpha \in (1,2)$,
\begin{equation}\label{eq:diL}
    \mathscr{L} u(x) = \int_{\mathbb{R}^{n}} [u(x+y)-u(x)-\mathbf{1}_{B_{1}(0)}(y)\langle y, \nabla u(x)\rangle ]\frac{K(\hat{y})}{|y|^{n+\alpha}}\,\mathrm{d}y,\quad x\in\mathbb{R}^n,
\end{equation}
where $\mathbf{1}_E$ denotes the indicator function of a set $E$, ${B_{1}(0)}$ is the unit ball of $\R^n$, $n\ge 2$, and $\hat{y}$ denotes the normalized vector $\frac{y}{|y|}$ for $y\ne 0$. If $K\equiv 1$, the operator $\mathscr{L}$ would be the infinitesimal generator of the standard $\alpha$-stable L\'{e}vy process $L^{\alpha}_t$ with L\'evy measure $|y|^{-n-\alpha}\,\,\mathrm{d}y$. As a slight modification, 
$K\in C^\infty(\mathbb{S}^{n-1})$ is assumed to be a nonnegative bounded even and uniformly nondegenerate function in $\mathbb{S}^{n-1}$, and then $\mathcal{L}$ is the generator of an $\alpha$-stable type process with L\'evy measure $K(\hat{y})|y|^{-n-\alpha}\,\,\mathrm{d}y$. See Section \ref{sec:jump diff} for more details on this connection, where the scaling $\eps^{\alpha-1}$ in front of the operator $\scrL$ is also explained. We refer to \cite{A2004} for the standard stochastic analysis of jump diffusion processes. 

Given an initial value $u_{0}$ belonging to $\BUC(\mathbb{R}^n)$, the space of bounded uniformly continuous functions, the above non-local HJB equation is well-posed and admits a unique viscosity solution, see e.g. \cite{GB97, HP98}. In the case of $\alpha = 2$, the equation reduces to the standard second order HJB equation, see, e.g. \cite{HP09}. The homogenization problem concerns the limit, as $\eps\to 0$, of the solution $u_\eps$ to the heterogeneous problem \eqref{HJB}.

The homogenization of Hamilton-Jacobi (HJ) equations was first studied in the seminal work of Lions, Papanicolaou and Varadhan \cite{LPV1987} in the periodic setting for first order equations. The stochastic homogenization of first order HJ equation was first established independently by Souganidis \cite{MR1697831}, and by Tarver and Rezakhanlou \cite{MR1756906}, for convex Hamiltonians using optimal control representations. The stochastic homogenization of second order HJB equations was first established independently by Lions and Souganidis \cite{LS05}, and by Kosygina, Rezakhanlou, and Varadhan \cite{KRV06}, for convex Hamiltonians. Armstrong and Tran \cite{AT15} developed PDE approach to stochastic homogenization of convex HJB equations without using optimal control, completing a program initiated by Lions and Souganidis \cite{MR2664465}. 

In recent years, the homogenization of non-local PDEs associated to jump processes has attracted many attention.  Periodic and stochastic homogenization of jump processes in heterogeneous environments with different settings were investigated by Kassmann, Piatnitski, and Zhizhina in \cite{KPZ19}, and by Chen, Chen, Kumagai, Wang in \cite{CCKW21, CCKW21R}. Periodic homogenization of HJB, with Hamiltonian defined via optimization of integral-differential operators, was studied by Ciomaga, Ghilli and Topp \cite{MR4373878}.
%The periodic homogenization of non-local HJB equations was established independently by Arisawa \cite{A09}, Bardi, Cesaroni, and Topp \cite{BCT20}.
In this paper, we extend the result and the method in \cite{KRV06}, and obtain stochastic homogenization for non-local HJB equations with vanishing non-local term in a stationary random environment. This equation arises from some stochastic control problems with jump diffusions in random environments \cite{OS07, HP98}.

\subsection*{Notations and assumptions.} Throughout the paper, $(\Omega, \mathscr{F},\mathbf{P})$ denotes the probability space on which the random Hamiltonian is defined, and it is assumed that $\Omega$ is a Polish space, and $\mathscr{F}$ is the completion of the corresponding Borel $\sigma$-algebra $\mathcal{B}(\Omega)$. For $y\in \R^n$ and $y\ne 0$, $\hat{y} := y/|y|$ denotes the unit vector in the direction of $y$.

The stationary ergodic random Hamiltonian $H(p,x,\omega)$ is constructed following the standard approach using a group of measure preserving transformations on $(\Omega, \mathscr{F},\mathbf{P})$. More precisely, we assume that there is a family $\{\tau_{y}:\Omega\rightarrow \Omega\}_{y\in\mathbb{R}^n}$ of transformation on $\Omega$ such that, for each $y$, $\tau_y$ preserves the measure $\mathbf{P}$, and, moreover, $\{\tau_y\}_{y\in \R^n}$ satisfies the group property: $\tau_0 = \mathrm{id}_\Omega$ is the identity mapping on $\Omega$, and $\tau_{x+y}=\tau_x\circ \tau_y$ for all $x,y \in \R^n$. 

We assume that the mapping $(y,\omega)\mapsto\tau_{y}\omega $ is jointly measurable from $(\mathbb{R}^{n}, \mathscr{B}(\mathbb{R}^{n}), m_{\R^n}) \times (\Omega,\mathscr{F}, \mathbf{P})$ to $ (\Omega,\mathscr{F}, \mathbf{P})$, where $\mathscr{B}(\mathbb{R}^{n})$ is the Borel $\sigma$-algebra of $\mathbb{R}^n$, and $m_{\R^n}$ is the Lebesgue measure. As a direct consequence, for any $\mathcal{B}(\Omega)$-measurable random variable $f$ on $\Omega$, $\widetilde f(y,\omega):=f(\tau_{y}\omega)$ is jointly measurable on $\R^n\times \Omega$.

We also assume that the measure-preserving transformation group $\{\tau_y\}_{y\in \R^d}$ is ergodic: if $A \in \mathscr{F}$ is invariant under this group of transformations, that is, if $\tau_y A := \{\tau_y \omega\,:\, \omega\in A\}$ equals $A$ for all $y\in \R^d$, then $\bP(A)$ must have value in $\{0,1\}$.

Now we specify our assumptions on the stochastic Hamiltonian $H(p,x,\omega)$. We assume that, for each fixed $p\in\mathbb{R}^n$, $H(p,\cdot,\omega)$ is a stationary ergodic random field on $(\Omega, \mathscr{F}, \bP)$.  
With the set-up of ergodic measure-preserving group of transformations $\{\tau_y\}$, $H$ is specified as follows. Let $H(p,\omega):\mathbb{R}^n \times \Omega \rightarrow \mathbb{R}$ be a random function on $\Omega$, we define $\widetilde{H}(p,x,\omega):=H(p,\tau_{x}\omega)$ for each $x\in \mathbb{R}^n$. The stationarity of $\widetilde H$ is easily seen from the relation
\begin{equation*}
    \widetilde{H}(p,x+y,\omega) = \widetilde{H}(p,x,\tau_y\omega).
\end{equation*}
and the fact that $\tau_y$ preserves measure. By an abuse of notation, we still write $\tilde{H}(p,x,\omega)$ as $H(p,x,\omega)$ in this paper.

The above assumptions on the set-up of the probability space and on the form of $H$  are always invoked and they are referred to as the ``standard stationary ergodic set-up''. 

For the homogenization of \eqref{HJB}, we need \emph{some} of the following hypotheses of $H(p,\omega)$ and of the initial data $u_0$:
\begin{itemize}
\item[{\bf (A1)}] $H$ is \emph{convex} and \emph{super-linear} in $p$. Moreover, there are positive constants $c_1$, $c_2$, $c_3$, $c_4$ and $\beta$, satisfying $1< \beta < \infty$, so that for all $p\in \mathbb{R}^n$ and $\omega\in\Omega$, 
\begin{equation}\label{A1H}
    c_1 |p|^{\beta} - c_2 \leq H(p,\omega) \leq c_3 |p|^{\beta} + c_4.
\end{equation}
This implies that the Lagrangian function $L(q,\omega)$ satisfies 
\begin{equation}\label{A2L}
    c'_1|q|^{\beta'} - c_4 \leq L(q,\omega) \leq c'_3 |q|^{\beta'} + c_2
\end{equation}
for $\beta' = \beta/(\beta-1)$ and some positive constants $c'_1,c'_3$.

\item[{\bf (A2)}] The transformation mapping $x\mapsto H(p,\tau_{x}\omega)$ is uniformly continuous. More precisely, for any $\ell >0$,
\begin{equation}
    \lim_{\delta\to 0}\sup_{|x| \leq \delta}\sup_{|p| \leq \ell} \underset{\omega \in \Omega}{\mathrm{ess}\sup} |H(p,\tau_{x}\omega) - H(p,\omega)| = 0.
\end{equation}
By Legendre transformation and the bounds \eqref{A2L}, the Lagrangian $L(q,\omega)$ satisfies the same type of uniform continuity in space: for any $\ell>0$,
\begin{equation}
    \lim_{\delta\to 0}\sup_{|x| \leq \delta}\sup_{|v| \leq \ell} \underset{\omega \in \Omega}{\mathrm{ess}\sup} |L(v,\tau_{x}\omega) - L(v,\omega)| = 0.
\end{equation}

\item[{\bf (A3)}] The initial value $u_0$ is uniformly continuous on $\mathbb{R}^n$. Note that this implies, for every $\delta>0$ there exists a constant $K_{\delta}>0$ so that for every $x,y \in\mathbb{R}^n$,
\begin{equation}\label{u0C}
    |u_0(x)-u_{0}(y)| \leq K_{\delta} |x-y|+ \delta.
\end{equation}

\item[{\bf (B)}] There exists a modulus of continuity $\gamma:(0,1)\to [0,\infty)$, satisfying $\gamma(\delta)\to 0$ as $\delta\to 0$ and a positive constant $M$, such that, for all $x\in B_\delta(0)$ with $\delta \in (0,1]$, and for all $\omega\in\Omega$,
\begin{equation}\label{AHd}
    H(p,\tau_{x}\omega) \geq (1 + \gamma(\delta))H((1+\gamma(\delta))^{-1}p,\omega) - M\gamma(\delta).
\end{equation}

\item[{\bf (C)}] Assumption {\bf (A1)} holds with $\beta > n$.
\end{itemize}

Throughout the paper, Assumptions ({\bf A1})({\bf A2})({\bf A3}), which are henceforth grouped together and called Assumption ({\bf A}), and the stationary ergodic set-up of $H$ above those assumptions are always invoked. In fact, because only qualitative homogenization is addressed in the paper, we may impose the stronger assumption ``$u_0$ is Lipschitz on $\R^n$'' without loss of generality; see Remark \ref{rem:BLip}. Those assumptions are somewhat standard for stochastic homogenization of HJB equations. 
Assumptions ({\bf B}) and ({\bf C}) are additional ones that will be needed in the construction of super correctors in the homogenization approach. As one may expect, assumption ({\bf C}) relates to the critical exponent for the application of Sobolev embedding. Assumption ({\bf B}) somehow treats the convexity of $H$ in $p$ with its continuity in $x$, and is most easily satisfied if $H$ is of the separation form $H(p,x)=G(p)+V(x)$.

The goal of this paper is to establish a qualitative stochastic homogenization result for the heterogeneous HJB equation \eqref{HJB}. This amounts to finding the \emph{effective Hamiltonian} $\overline{H}(p)$, a convex real-valued function on $\mathbb{R}^n$ with the salient feature that it is deterministic (does not depend on $\omega$) and is homogeneous in space (does not depend on $x$), such that as $\eps\to 0$ the solution $u_\eps$ of \eqref{HJB} converges locally uniformly to the \emph{homogenized} Hamilton-Jacobi equation
\begin{equation}\label{hHJB}
  \left\{
   \begin{aligned}
   & \partial_t \overline{u}(t,x) =  \overline{H}(\nabla \overline{u}(t,x)),\quad & &(t,x) \in [0,\infty)\times\mathbb{R}^n, \\
   &  \overline{u}(0,x) = u_0(x), \quad & &x\in \mathbb{R}^n.
   \end{aligned}
   \right.
\end{equation}
We emphasize that both the characterization of $\overline{H}$ and the proof of almost sure locally uniform convergence of $u_\eps \to \ol{u}$ are problems to be solved. %and is given by (\ref{EffH}) in Section 2.

\subsection{Main result} Now we state our main result in this paper as follows.
\begin{thm}
\label{thm:main}
Under the standard stationary ergodicity set-up and Assumptions {\upshape({\bf A})}, together with either Assumption {\upshape({\bf B})} or Assumption {\upshape({\bf C})}, let $u_{\epsilon}(t,x,\omega)$ be the solution of the HJB equation (\ref{HJB}), with $\alpha\in (1,2)$, and $\overline u(t,x)$ be the solution of the homogenized HJ equation \eqref{hHJB}, with $\ol{H}$ defined in \eqref{EffH}. Then for every $R,T>0$ and $\bP$-almost surely we have
\begin{equation}\label{eq:main}
    \lim_{\epsilon\rightarrow 0}\sup_{t\in[0,T]}\sup_{|x|\leq R}|u_{\epsilon}(t,x,\omega) - \overline u(t,x)|= 0 .
\end{equation}
\end{thm}

%This homogenization result is coming from lower bounds in Section \ref{sec:lowerbdd}, and upper bounds in Section \ref{sec:proofUB}.
The result above establishes the qualitative stochastic homogenization of HJB equations with vanishing non-local term. It is known, as seen from the case of local equations studied in \cite{KRV06,LS05,AT15,JST17}, the vanishing ``higher-order'' term introduces extra difficulty for homogenization. In this paper, this additional term is non-local and it is very natural considering the associated stochastic optimal control problem. As far as we know, this is the first work addressing stochastic homogenization in this setting. 

We end this Introduction with several remarks. First, in this paper, we adapt the method of \cite{KRV06} to the non-local setting. The method is based on control representation and on abstract environmental processes, but does not use much viscosity solution theory. On the other hand, such PDE theory for non-local HJB equation involving generators of jump-diffusion processes has been studied intensively recently, see e.g., \cite{MR2422079,MR4024555,MR5058962}. It would be interesting to develop a homogenization method based more on viscosity solution theory. Second, it is known that homogenization is in general more difficult to establish for highly oscillating dynamic random environment, that is, when the coefficients of HJ/HJB equations have heterogeneous random variations in time variable as well; see, e.g., \cite{KV-HJB,JST17,MR3817561}. It is natural to pursue this direction for the non-local equation like \eqref{HJB}. Finally, there is always the more challenging question of quantifying the convergence rate. In the convex setting, we point out the optimal rate for periodic homogenization of HJ equation proved in \cite{MR4946874}, and a recent result \cite{guo2026quantificationergodicityhamiltonjacobiequations} in a dynamic random setting. We will study such problems in future works.

\subsection*{Organization of the paper} In the next section, we present the stochastic optimal control problem that leads to the HJB equations \eqref{HJB}, in which the noise part is a pure jump process. We also rewrite the non-local operator $\mathscr{L}$ into the form of the divergence of another non-local operator acting on the gradient. In section \ref{sec:envpro} we follow the approach of Kosygina, Rezakhanlou and Varadhan \cite{KRV06} to define the environmental process in the probability space lifted from jump processes on $\R^n$, and define the homogenized Lagrangian and Hamiltonian. To prove that the effective Hamiltonian is the right one, i.e., to establish the stochastic homogenization result, we show in section \ref{sec:lowerbdd} that the lower bound part of the convergence follows quite directly from the stochastic optimal control formula and the ergodic theorem, which also explains the definition of the effective Hamiltonian. We also establish some interesting regularity results of the solution to \eqref{HJB}. The upper bound part of the convergence is more involved. In section \ref{sec:convex} we follow again the method of \cite{KRV06} to construct proper ``super-correctors". This is a main novelty of the paper because we need to overcome the difficulty caused by the non-local operator. The ``divergence form" found in section \ref{sec:envpro} plays an important role. Finally, in section \ref{sec:proofUB}, we establish the upper bound part and complete the proof of stochastic homogenization under two different set-up assumptions.

%%%%%%%%%%%
\section{Preliminaries on the non-local operator} 

%%%%
\subsection{The stochastic differential equation with L\'{e}vy noise}
\label{sec:jump diff}

The non-local HJB equation (\ref{HJB}) arises naturally from a stochastic control problem driven by a jump diffusion process. Let $\mathscr{C}$ be the control set defined by 
\begin{equation}
\label{eq:controlset}
\mathscr{C} = \{c \in L^\infty([0,T]\times \R^n;\R^n) \,:\, c \text{ is Lipschitz with respect to $x$}\}.
\end{equation}
That is, $\mathscr{C}$ consists of bounded maps $c(t,x)$ from $[0,T]\times\mathbb{R}^n$ to $\mathbb{R}^n$ that are Lipschitz continuous in $x$. 
Denote by $X_t$ the jump diffusion process on $\mathbb{R}^n$ starting from $x\in\mathbb{R}^n$ at $t=0$ and subjected to a drift vector field $c \in \mathscr{C}$ and a L\'evy noise. More precisely, $X_t = X_t^c$ is the solution of the following stochastic differential equation:
\begin{equation}\label{SDElevy}
  \left\{
   \begin{aligned}
   & \upd X_t = c(t,X_t)\mathrm{d}t +  \mathrm{d}L^{\alpha}_t,\quad & t\in [0,T], \ X_t \in \mathbb{R}^n, \\
   &  X_0 = x, \quad & x\in \mathbb{R}^n,
   \end{aligned}
   \right.
\end{equation}
where $L^{\alpha}_t$, with $\alpha \in (1,2)$, is an $\alpha$-stable L\'{e}vy process with generator $\mathscr{L}$. The characteristic function of $L^\alpha_t$ is given by
\begin{equation*}
    \mathbb{E}\exp(\myi\langle \xi, L^{\alpha}_t \rangle) = e^{t \psi(\xi)}
\end{equation*}
with 
\begin{equation*}
    \psi(\xi) = \int_{\mathbb{R}^n} \left(e^{\myi \langle y, \xi \rangle} -1 - \myi\mathbf{1}_{B_{1}(0)}(y)\langle y, \xi \rangle  \right) \nu(\mathrm{d}y)
\end{equation*}
where $\nu(\mathrm{d}y) = K(\hat{y})\,\mathrm{d}y/|y|^{n+\alpha}$ is the L\'{e}vy measure of $L^{\alpha}_{t}$. Here, $\hat{y}$ is the unit vector in the direction of $y$ and $K: \mathbb{S}^{n-1} \to (0,\infty)$ is a bounded positive function defined on the unit sphere $\mathbb{S}^{n-1}$ and satisfies $K(\theta) = K(-\theta)$ and hence the centering condition
\begin{equation}
    \label{eq:Ksym}
    \int_{\mathbb{S}^{n-1}} \theta K(\theta)\,\upd\sigma(\theta) = 0.
\end{equation}
Here, $\upd\sigma(\theta)$ is the surface measure on the sphere $\mathbb{S}^{n-1}$. In the rest of the paper, however, the simplified notation $\upd\theta$ is used instead for it. The special case of $K\equiv 1$ corresponds to the standard $\alpha$-stable L\'evy process. The centering condition above for a more general $K$ allows one to verify that the resulting process $L^{\alpha}_{t}$ remains self-similar:
\begin{equation}
    \label{eq:Ltass}
    L^{\alpha}_{\epsilon^{-1}t} \eqdistr \epsilon^{-1/\alpha} L^{\alpha}_{t}.
\end{equation}

As usual we always consider a version of $L_t^\alpha$ such that its paths are almost surely in the \emph{C\`adl\`ag} space $\mathbb{D}([0,\infty);\mathbb{R}^{n})$, that is the space of real-valued right continuous functions with left limits. We then define the Poisson random measure associated to the pure jump process $\Delta L^{\alpha}_t =L^{\alpha}_{t}-L^{\alpha}_{t^-}$ as
\begin{equation*}
    N(t,B) := \sharp \{s\in [0,t]:\Delta L^{\alpha}_{s} \in B\}, \quad t\geq 0, B\in \mathscr{B}(\mathbb{R}^n\backslash \{0\}).
\end{equation*}
Note that $N(t,\cdot)$ is a random measure on $\mathbb{R}^n$ determined by the paths of $L_t^\alpha$. We have $\mathbb{E}[N(\mathrm{d}t,\mathrm{d}y)] = \mathrm{d}t\,\nu(\mathrm{d}y)$ and define the corresponding compensated Poisson random measure 
\begin{equation*}
   \widetilde{N}(\mathrm{d}t,\mathrm{d}y)=N(\mathrm{d}t,\mathrm{d}y)-\mathrm{d}t\,\nu(\mathrm{d}y) .
\end{equation*}
Then by L\'{e}vy-It\^{o} decomposition theorem, path-wisely $L^{\alpha}_t$ can be described as
\begin{equation*}
    L^{\alpha}_t = \int_0^t \int_{0<|y|<1}y\widetilde{N}(\mathrm{d}t,\mathrm{d}y)+ \int_0^t\int_{|y|\geq 1}y N(\mathrm{d}t,\mathrm{d}y).
\end{equation*}
Thus the SDE (\ref{SDElevy}) is equivalent to
\begin{equation}\label{Lequu}
 \mathrm{d}X_t = c(t,X_{t})\mathrm{d}t + \int_{|y|<1}y\widetilde{N}(\mathrm{d}t,\,\mathrm{d}y) + \int_{|y|\geq 1}yN(\mathrm{d}t,\mathrm{d}y).
\end{equation}
For every $V \in C^{2}(\mathbb{R}^n)$, we have the following It\^{o} formula (see Theorem 4.4.7 of \cite{A2004}) for the jump diffusion process $X_{t}$
\begin{equation}
\label{eq:ItoC2}
\begin{aligned}
    V(X_{t}) - V(x) =& \int_{0}^{ t} \langle c(s,X_s), \nabla V(X_s)\rangle \mathrm{d}s  + \int_{0}^{t} \int_{\mathbb{R}^n} [V(X_{s^{-}}+ y)-V(X_{s^{-}})]\widetilde{N}(\mathrm{d}s,\mathrm{d}y) \\
    & +  \int_{0}^{t} \int_{\mathbb{R}^n}[V(X_{s^{-}}+y)-V(X_{s^{-}}) - \mathbf{1}_{B_{1}(0)}(y)\langle y, \nabla V(X_{s^{-}}) \rangle ]\nu(\mathrm{d}y)\mathrm{d}s.
\end{aligned}
\end{equation}
The sample paths of $X_{t}$ also belong to the Skorokhod space $\mathbb{D}([0,\infty);\mathbb{R}^{n})$. Define the Skorokhod metric $d(\vartheta_1, \vartheta_2)$ on $\mathbb{D}([0,\infty);\mathbb{R}^{n})$ by
\begin{equation*}
d(\vartheta_1, \vartheta_2)=\inf_{\lambda \in \Lambda} \left\{ \sup_{t \in [0,\infty)}|\lambda(t) - t| + \sup_{t \in [0,\infty)}\left| \vartheta_{1}(\lambda(t))-\vartheta_{2}(t)\right| \right\}.
\end{equation*}
where $\Lambda$ denotes the set of all strictly increasing continuous bijections from $[0,\infty)$ to itself.
Then $\mathbb{D}([0,\infty);\mathbb{R}^{n})$ with the Skorokhod metric is a Polish space; see e.g. \cite[Section VI.1]{JS13}. In the rest of the paper, we denote by $Q^{c}_{x}$ the law of the jump diffusion process $X_{t}$ starting from $x$ determined by the SDE \eqref{Lequu}, which is a probability measure on the Skorokhod space $\mathbb{D}([0,\infty);\mathbb{R}^{n})$.

%%%%%%%%%%%
\subsection{A divergence form characterization of $\mathscr{L}$} 

For the non-local operator $\mathscr{L}$, it is more or less well known that $\mathscr{L}u$ can be written as the divergence of another non-local operator $\mathcal{R}$ acting on the gradient $\nabla u$. To find this out, for a curl-free vector field $v$ on $\R^n$, we define
\begin{equation}
    \label{eq:nldivRn}
    \mathscr{J}v(x) = \int_{\R^n} \left(\int_{x\rightsquigarrow x+y} \langle v(z(s)) - \mathbf{1}_B(y) v(x), \mathrm{d}z(s)\rangle\right)\frac{K(\hat{y})}{|y|^{n+\alpha}}\,\,\mathrm{d}y.
\end{equation}
The above can be thought as a non-local divergence of $v$. The line integral inside is along a smooth simple curve $z(\cdot)$ connecting $x$ to $x+y$. Because $v$ is curl-free, the choice of $z(\cdot)$ can be arbitrary. From this definition one can check that
\begin{equation}
\label{eq:JLRn}
\mathscr{L} u = \mathscr{J}(\nabla u)
\end{equation}
holds for all scalar functions $u\in C^2_{\rm b}(\R^n)$.

Following the method used in \cite[Proposition 3.3]{MR2944369}, we can characterize $\mathscr{L}$ using Fourier transform. Let $u$ be a function of Schwartz class. A direct computation shows the Fourier transform of $\mathscr{L}u$ is given by
\begin{align}
    \mathcal{F}(\mathscr{L} u)(\xi) = & \int_{\mathbb{R}^n} (\cos(\xi \cdot y)-1) \frac{K(\hat{y})}{|y|^{n+\alpha}}\mathcal{F}u(\xi) \,\mathrm{d}y \nonumber \\
    = & \int_{\mathbb{S}^{n-1}}\int_{0}^{\infty}  (\cos(\xi \cdot r\theta )  -1) \frac{K(\theta)}{r^{\alpha+1}}\mathcal{F}u(\xi) \mathrm{d}r \mathrm{d}\theta,
\end{align}
where we used the fact that the mapping $\theta \mapsto K(\theta)$ is even. Setting $\zeta = r|\xi \cdot \theta|$, we obtain
\begin{align}
    \mathcal{F}(\mathscr{L} u)(\xi)  = &  \mathcal{F}u(\xi)\int_{\mathbb{S}^{n-1}} |\xi\cdot\theta|^\alpha \int_{\R_+}   \frac{\cos(\zeta)-1}{|\zeta|^{1+\alpha}} \mathrm{d}\zeta K(\theta)\mathrm{d}\theta \nonumber \\
    = & -C(n,\alpha)\mathcal{F}u(\xi)\int_{\mathbb{S}^{n-1}}|\xi\cdot\theta|^{\alpha} K(\theta) \mathrm{d}\theta.
\end{align}
where $C(n,\alpha)= \int_{\R_+} [1-\cos(\zeta )]/|\zeta|^{1+\alpha} \mathrm{d}\zeta $ is a constant depending on $n$ and $\alpha$.

In view of \eqref{eq:JLRn}, we have $\mathscr{J}v = \nabla \cdot (\mathcal{R}v)$ for curl-free vector field $v$, if the operator $\mathcal{R}$ acting on each component $v_j$ of $v$ is defined by its Fourier symbol as follows:
\begin{equation}\label{Fsym}
    \mathcal{F}(\mathcal{R}v_j)(\xi) = P(\xi) \mathcal{F}{v_j}(\xi), \quad \text{with} \quad P(\xi) := C(n,\alpha)\int_{\mathbb{S}^{n-1}}|\xi|^{-2}|\xi\cdot\theta|^{\alpha} K(\theta) \mathrm{d}\theta.
\end{equation}
%From its symbol we see that $\mathcal{R} = \mathscr{L} \Delta^{-1}$. 
The Fourier symbol $P(\xi)$ of $\mathcal{R}$ is homogeneous of order $\alpha-2$. When $K\equiv 1$, it is well known that the operator $\mathcal{R}$ becomes the Riesz potential and its integral kernel is proportional to the function $|x|^{2-n-\alpha}$. For general $K$ that is a positive and smooth function on $\mathbb{S}^{n-1}$, from \eqref{Fsym} we see that $P(\xi)$ is still homogeneous of degree $\alpha-2$, is locally integrable and is smooth on $\R^n\setminus \{0\}$. As a homogeneous tempered distribution, its (inverse) Fourier transform is a homogeneous distribution of degree $-n-\alpha+2$. On the other hand, by Theorem 2.4.8 of \cite{Grafakos2008}, this remains a function smooth on $\R^n\setminus \{0\}$. Hence, we see that $\mathcal{R}$ admits an integral kernel that is homogeneous of degree $-n-\alpha+2$ and is smooth on $\R^n\setminus \{0\}$. In other words, 
\begin{equation}
\label{eq:cRdef}
    \mathcal{R}g(x) = \int_{\R^n} R(x-y)g(y)\,\,\mathrm{d}y, \qquad R(x) = \frac{\Theta(\hat{x})}{|x|^{n+\alpha-2}},
\end{equation}
where $\hat{x} = x/|x|$ denotes the direction of $x$, and $\Theta$ is  smooth on $\mathbb{S}^{n-1}$ and in particular uniformly bounded. As a result, $\mathcal{R}$ has an integral kernel that is integrable near $0$. In other words, $\mathcal{R}$ is a regular integral operator.

% From the above characterization of $\mathscr{L}$, we find that its adjoint $\mathscr{L}^*$ is defined by
% \begin{equation*}
%     \mathscr{L}^* \phi = \nabla \cdot (\mathcal{R}^* \nabla \phi),
% \end{equation*}
% where $\mathcal{R}^*$ is the adjoint of $\mathcal{R}$ defined by
% \begin{equation*}
%     \mathcal{R}^* f(x) = \int_{\R^n} R^*(x-y)f(y)\,\upd y = \int_{\R^n} R(y-x)f(y)\,\upd y.
% \end{equation*}
% Since we assume $K(y)=K(-y)$, both $\mathcal{R}$ and $\mathscr{L}$ are self-adjoint.
%%%%%%%%%%%
\subsection{The non-local Hamilton-Jacobi-Bellman equation}
The non-local HJB equation \eqref{HJB} naturally arises from a stochastic control problem, as we now describe. 

Let $L(v,\omega)$ be the random Lagrangian function which is defined as the Legendre transformation of the random Hamiltonian $H(p,\omega)$, that is,
\begin{equation}
    L(v,\omega):= \sup_{p\in\mathbb{R}^n}[\langle p, v \rangle - H(p,\omega)], \quad v\in\mathbb{R}^n,\ \omega\in\Omega.
\end{equation}
Under the assumptions on $H$, for each fixed $v\in\mathbb{R}^n$, the random field $L(v,x,\omega) := L(v,\tau_x\omega)$ is also a stationary ergodic random field on $(\Omega, \mathscr{F}, \bP)$.
For each $g(x)\in \BUC(\mathbb{R}^n)$, $c(t,x)\in\mathscr{C}$ and $\omega\in \Omega$, consider the gain functional
\begin{equation}
    \mathbb{E}^{Q^c_x}\left( g(X_t) - \int^{t}_0 L(c(s,X_s),\tau_{X_s}\omega)\mathrm{d}s \right),
\end{equation}
and the optimal control problem of maximizing it with respect to the control $c \in \mathscr{C}$. Here $\mathbb{E}^{Q^{c}_{x}}$ denotes the expectation with respect to the path measure $Q^{c}_{x}$.

Let $u$ be the value function of this optimal control problem, that is,
\begin{equation}
\label{eq:HJBcontrol}
    u(t,x,\omega):= \sup_{c\in\mathscr{C}}\mathbb{E}^{Q^{c}_{x}}\left( g(X_t) - \int^{t}_0 L(c(s,X_s),\tau_{X_s}\omega)\mathrm{d}s \right).
\end{equation}
As in the setting where $\mathscr{L}$ is the Laplacian, it can be shown that $u(\cdot,\cdot,\omega)$ is the unique viscosity solution to the following Cauchy problem of HJB equation:
\begin{equation}\label{HJBu}
  \left\{
   \begin{aligned}
   & \partial_t u(t,x,\omega) =  \mathscr{L} u(t,x) + H(\nabla u(t,x), \tau_{x}\omega ),\quad & (t,x) \in [0,\infty)\times\mathbb{R}^n, \\
   &  u(0,x) = g(x), \quad & x\in \mathbb{R}^n.
   \end{aligned}
   \right.
\end{equation}
We refer to \cite[Theorem 3.4]{GB97} or \cite[Theorem 4.1]{LP09} for the proof of this fact. In this paper we follow the approach of \cite{KRV06,KV-HJB} and do not use the theory of viscosity solutions in any essential way, so we omit the proof of this fact.

The HJB equation above enjoys nice properties with respect to the shifting and scaling of the spatial variable. Let $u(t,x,\omega)$ be the solution to \eqref{HJBu} with initial data $g$. Fix $y\in \R^n$, by inspecting \eqref{HJBu} (or more rigorously, by verification using \eqref{eq:HJBcontrol}), one sees that the function $u^y(t,x,\omega) :=  u(t,x+y,\omega)$ is the solution to \eqref{HJBu} with shifted environment $\omega'=\tau_y\omega$, and $u^y$ satisfies the shifted initial condition $u^y(0,x,\omega) = g^y(x) := g(x+y)$. %If we define $u_{0}^{y}(x) = u_{0}(x+y)$, then the solution of (\ref{HJBu}) with initial data $u(0,x)=u_{0}^{y}(x)$ and $\omega' = \tau_{y}\omega$ is given by
Let $S(t,\omega)$ denote the mapping $g\mapsto S(t,\omega)g = u(t,\cdot,\omega)$ where $u$ is the solution to \eqref{HJBu}. It follows that
\begin{equation}\label{Trans}
    \left(S(t,\omega)g\right)(x+y) = \left(S(t,\tau_y \omega)g^y\right)(x).
\end{equation}

For the scaling in the spatial variable, we check that, given $\eps \in (0,1)$, the solution to \eqref{HJB} is given by
\begin{equation}
\label{eq:uepsscale}
    u_\eps(t,x,\omega) = \eps\left(S(\tfrac{t}{\eps},\omega) (\eps^{-1}u_0(\eps \,\cdot))\right)(\frac{x}{\eps}).
\end{equation}
Here the right hand side is obtained by solving \eqref{HJBu} with initial data $g(x) = \eps^{-1}u_0(\eps x)$ and then rescaling the solution $u(t,x,\omega)$ to $\eps u(t/\eps,x/\eps,\omega)$. Formally, this can be verified by checking
\begin{equation*}
\left(\mathscr{L} u_\epsilon\right)(t,x) = \epsilon^{1-\alpha} (\mathscr{L} u)(\epsilon^{-1}t, \epsilon^{-1}x).
\end{equation*}
More rigorously, this can be verified via a direct scaling in the stochastic optimal control formula \eqref{HJBu}. Indeed, by definition, 
\begin{equation*}
u_\eps(t,x,\omega) = \eps u(\eps^{-1}t, \eps^{-1}x,\omega) = \sup_{c\in \mathscr{C}} \mathbb{E}^{Q^c_{x/\eps}}\left(u_0(\eps X_{t/\eps}) - \eps \int_0^{t/\eps} L(c(s,X_s),\tau_{X_s} \omega)\,\mathrm{d}s\right),
\end{equation*}
where given $c\in \mathscr{C}$, $X_s$ solves the $L^\alpha$-driven SDE \eqref{SDElevy} with drift $c$ and with initial position $x/\eps$. Consider the rescaled process $X^{\epsilon}_{t}:= \epsilon X_{t/\epsilon}$. Then $X^\eps_t$ satisfies
\begin{equation}
\label{eq:Xepsc}
    \mathrm{d}X^{\epsilon}_{t} = c(\epsilon^{-1}t,\epsilon^{-1}X^{\epsilon}_{t})\mathrm{d}t  + \epsilon \mathrm{d}L^{\alpha}_{t/\epsilon}, \quad X^{\epsilon}_{0} =  x.
\end{equation}
For any $V\in C^2(\R^n)$, the function $V^\eps(x):= V(\eps x)$ remains in $C^2(\R^n)$. Applying It\^{o} formula \eqref{eq:ItoC2} to $V^\eps$ for the jump diffusion process $X_\bullet$, we have
\begin{equation*}
\begin{aligned}
    & V(X^{\epsilon}_{t}) - V(x) = V^\eps(X_{t/\eps}) - V^\eps(X_0)\\
    =&\int_{0}^{t/\eps} \langle c(s, X_{s}), \nabla V^\eps(X_{s}, \omega)\rangle \mathrm{d}s  + \int_{0}^{t/\eps} \int_{\mathbb{R}^n} [V^\eps(X_{s^{-}}+ y)-V^\eps(X_{s^{-}})]\widetilde{N}(\mathrm{d}s,\,\mathrm{d}y) \nonumber \\
    & +  \int_{0}^{t/\eps} \int_{\mathbb{R}^n}[V^\eps(X_{s^{-}}+y)-V^\eps(X_{s^{-}}) - \mathbf{1}_{B_{1}(0)}(y)\langle y,\nabla V^\eps(X_{s^{-}}) \rangle ]\nu(\,\mathrm{d}y)\mathrm{d}s.
    \end{aligned}
    \end{equation*}
Let $N_\epsilon$ denote the Poisson random measure of the rescaled L\'evy process \(L_t^{\alpha,\epsilon}:=\epsilon L^\alpha_{t/\epsilon}\), so that \( N_\epsilon(\dd s,\dd y) =N(\epsilon^{-1}\dd s,\epsilon^{-1} \dd y)\).
Its compensator is $\epsilon^{\alpha-1}\mathrm{d}t\,\nu(\mathrm{d}y)$; denote $\widetilde N_\epsilon$ for the associated compensated measure.

For the integrals above, we use the change of variables $\eps s \to s$, $\eps y \to y$. In view of the fact $(\nabla V^\eps)(x) = \eps (\nabla V)(\eps x)$, the symmetry condition \eqref{eq:Ksym} and the definition of $X^\eps_\bullet$, we can transfer the result above to the function $f$ and the process $X^\eps_\bullet$, and get
    \begin{equation}
    \label{eq:rescaledIto}
    \begin{aligned}
  V(X^\eps_t) -V(X_0^\eps)  =& \int_{0}^{ t} \langle 
    c(\epsilon^{-1}s,\epsilon^{-1}X^{\epsilon}_{s}), \nabla V(X^{\epsilon}_{s})\rangle \mathrm{d}s \\
    & + \int_{0}^{t} \int_{\mathbb{R}^n} [V(X^{\epsilon}_{s^{-}}+ y)-V(X^{\epsilon}_{s^{-}})]\widetilde{N}_{\epsilon}(\mathrm{d}t,\mathrm{d}y)\\
    & + \epsilon^{\alpha-1 }\int_{0}^{t} \int_{\mathbb{R}^n}[V(X^{\epsilon}_{s^{-}}+y)-V(X^{\epsilon}_{s^{-}}) -\mathbf{1}_{B_{1}(0)}(y)\langle y, \nabla V(X^{\epsilon}_{s^{-}}) \rangle ]\nu(\,\mathrm{d}y)\mathrm{d}s.
\end{aligned}
\end{equation}
It follows that $X^{\epsilon}_{t}$ is a jump diffusion process with generator $\epsilon^{\alpha-1} \mathscr{L} + c(\epsilon^{-1}t, \epsilon^{-1}x)\cdot\nabla$, starting from $x$ at time $0$. Let $Q^{c,\epsilon}_{x}$ denote the law of this process, we then obtain the variational formula for $u_{\epsilon}(t,x,\omega)$:
\begin{equation}
\label{vf}
\begin{aligned}
    u_{\epsilon}(t,x,\omega) = & \sup_{c\in\mathscr{C}} \mathbb{E}^{Q^{c}_{x/\epsilon}}\left( u_{0}(\epsilon X_{t/\epsilon}) - \epsilon\int^{t/\epsilon}_0 L(c(s,X_s),\tau_{X_{s}}\omega)\mathrm{d}s \right)\\
    = & \sup_{c\in\mathscr{C}} \mathbb{E}^{Q^{c,\epsilon}_{x}}\left( u_{0}(X^{\epsilon}_{t}) - \xi_{\epsilon}(t) \right).
\end{aligned}
\end{equation}
where 
\begin{equation}
\label{eq:xiepsdef}
    \xi_{\epsilon}(t) %= \epsilon\int^{t/\epsilon}_0 L(c(s,X_s),\tau_{X_{s}}\omega)\mathrm{d}s 
    = \int^{t}_0 L(c(\epsilon^{-1}s, \epsilon^{-1}X^{\epsilon}_s),\tau_{\epsilon^{-1}X^{\epsilon}_{s}}\omega)\mathrm{d}s.
\end{equation}
Since the second line in \eqref{vf} is precisely the optimal control formula for the solution of \eqref{HJB}, we verify again $u_\eps(t,x,\omega)=\eps u(\eps^{-1}t,\eps^{-1}x,\omega)$. That is, \eqref{eq:uepsscale} holds.

%%%%%%%%%%%
\begin{rem}\label{rem:BLip} As an immediate consequence of the representation formula \eqref{vf}, for fixed $\omega \in \Omega$, the solution mapping $S_\eps(t,\omega): u_0 \mapsto u(t,\cdot)$ is a semigroup and satisfies
\begin{equation*}
    \sup_{x\in \R^n} \left| S_\eps(t,\omega)u_0(x) - S_\eps(t,\omega)v_0(x) \right| \le \|u_0-v_0\|_\infty.
\end{equation*}
This can also be obtained by the comparison principle for non-local HJB equations (e.g., Proposition 3.1 of \cite{MR3385169}). This observation allows us to reduce the class of initial data satisfying ({\bf A3}) to a strictly smaller class satisfying:
\begin{itemize}
\item[({\bf A3$'$})] The initial value $u_0$ is bounded and Lipschitz (this is stronger than being BUC):
\begin{equation}\label{BLip}
    \|u_0\|_\infty < \infty \, \text{and $\exists K < \infty$ such that}\, |u_0(x)-u_0(y)|\le K|x-y|, \; \forall x,y.
\end{equation}
\end{itemize}
Indeed, if $u_0 \in \mathrm{BUC}$, then there is a sequence of functions $\{u_0^{k}\}_{k\in \mathbb{N}}$ such that $u_0^k$ satisfies ({\bf A3$'$}) for all $k$ (with constant $K$ depending on $k$) and $\|u_0^k-u_0\|_\infty$ converges to zero. Let $u^k_\eps(t,x)$ be the solution to \eqref{HJB} with initial value $u_0^k$, then by the above comparison result we have
\begin{equation*}
    |u^k_\eps(t,x;\omega)-u_\eps(t,x;\omega)| \le \|u_0 - u_0^k\|_\infty.
\end{equation*}
The above inequality is uniform in $t,x$ and $\omega$. Similar estimates hold for the solutions to \eqref{hHJB} with initial data $u_0$ and $u_0^k$. As a result, to prove Theorem \ref{thm:main}, we can first establish the convergence of $u^k_\eps$ to $\overline{u}^k$ and then let $k\to \infty$. In other words, in Theorem \ref{thm:main} we may assume that $u_0$ satisfies ({\bf A3$'$}) instead.
\end{rem}

We end this section by giving an example of $(\Omega,\scrF,\mathbf{P})$ that satisfies our assumptions in the paper. We take an example from \cite{ATY15} which is motivated by the classical mechanics Hamiltonian $H(p,x)=\frac12|p|^2 + V(x)$ where $V$ is the potential energy field. We can take 
\begin{equation*}
\Omega = \{V\in \mathrm{UC}(\R^n)\,:\, |V(x)|\le M \text{ for all $x$}\},
\end{equation*}
that is, the space of uniformly continuous functions on $\R^n$ uniformly bounded by a constant $M>0$. This is the sample space of the potential field $V$ and is equipped with the topology of locally uniform convergence. Then $\Omega$ is a Polish space. For the translation group $\{\tau_x \,:\, x\in \R^n\}$, we choose $\tau_x V(\cdot) = V(\cdot+x)$ to be the spatial translation of functions. Then it is more or less standard that one can choose a probability measure $\mathbf{P}$ on $\Omega$, for example induced from some Gaussian random field on $\R^n$ or from Poisson point process (see e.g. \cite{S98, GHV2018}), so that $\{\tau_x\}_{x\in \R^n}$ is measure-preserving and is ergodic. Now $\omega\in \Omega$ denotes a realization of $V(\cdot)$, and $H(p,\omega) =\frac12|p|^2+\omega(0)$.

%%%%%%%%%%%
\section{The definition of effective Hamiltonian}
\label{sec:envpro}

%%%
\subsection{Environmental process induced by jump-diffusion processes}

Now we claim the dense property of $L^{p}(\Omega,\mathscr{F},\mathbf{P})$ with $1 \leq p \leq \infty$.

\begin{lem}\label{Mollify}
For every $1\leq  p < \infty$, the set 
\begin{equation*}
    \{f\in L^{p}(\Omega) \;:\; x \mapsto f(\tau_{x}\omega) \in C^{\infty}(\mathbb{R}^n) \text{ for almost surely } \omega \in \Omega\}
\end{equation*}
is dense in $L^{p}(\Omega)$. For $p=\infty$, the set above is dense in the weak-$*$ topology in $L^\infty(\Omega)$.
\end{lem}

\begin{proof}
Consider a smooth even function $\rho(y)$ so that $\text{supp}\,\rho \subset B_{1}(0)$, $\rho(y) \geq 0$, and $ \int_{\mathbb{R}^n} \rho (x)\,\mathrm{d}x =1$. For every $f \in L^{1}(\Omega,\mathscr{F},\mathbf{P})$, we mollify $f$ by the convolution 
\begin{equation}
\label{eq:rvmolify}
    f_{\delta}(\omega)= \int_{\mathbb{R}^n} f(\tau_{y}\omega)\rho_{\delta}(y)\,\mathrm{d}y, 
\end{equation}
where $\rho_{\delta}(y) = \delta^{-n}\rho(y/\delta)$ for every $\delta \in (0,1)$. Then for every $x \in \mathbb{R}^n$,
\begin{equation*}
    f_{\delta}(\tau_{x}\omega)= \int_{\mathbb{R}^n} f(\tau_{x+y}\omega)\rho_{\delta}(y)\,\mathrm{d}y = \int_{\mathbb{R}^n} f(\tau_{y}\omega)\rho_{\delta}(y-x)\,\mathrm{d}y.
\end{equation*}
Thus for almost surely $\omega \in \Omega$, $x \mapsto f_{\delta}(\tau_{x}\omega) \in C^{\infty}(\mathbb{R}^n)$. Moreover,
\begin{equation*}
    f_{\delta}(\omega) - f(\omega) = \int_{B_{1}(0)} \rho(y)(f(\tau_{\delta y}\omega) - f(\omega))\,\mathrm{d}y.
\end{equation*}
Then by Fubini theorem, for all $ 1 \leq p < \infty$,
\begin{align*}
    \|f_{\delta} - f\|^{p}_{L^{p}(\Omega,\mathscr{F},\mathbf{P})} = & \int_{\Omega}|f_{\delta}(\omega) - f(\omega)|^{p}\mathrm{d}\mathbf{P} \\
    \leq & \int_{B_{1}(0)} \rho(y) \int_{\Omega} |f(\tau_{\delta y}\omega) - f(\omega)|^{p}\mathrm{d}\mathbf{P} \,\mathrm{d}y.%\\
    %\leq & \sup_{y \in B_{1}(0)}\int_{\Omega} |f(\tau_{\delta y}\omega) - f(\omega)|^{p}\mathrm{d}\mathbf{P}.
\end{align*}
Recall that the transition group $\tau_{y}$ acts on $L^p(\Omega)$ by $(\tau_y f)(\omega) = f(\tau_y \omega)$ and is strongly continuous; see e.g. \cite[section 7.1]{JKO12}. It follows that the above vanishes as $\delta\to 0$. This completes the proof for $p\in [1,\infty)$.

For $p=\infty$, take any $g\in L^1(\Omega)$, we note that
\begin{equation*}
\begin{aligned}
\int_\Omega f_\delta(\omega) g(\omega)\,\dd \mathbf{P} &= \int_\Omega \left(\int_{\R^n} f(\tau_{\delta y}\omega) \rho(y)\,\dd y\right) g(\omega)\,\dd \mathbf{P} \\
&=  \int_{\R^n}\left(\int_\Omega f(\tau_{\delta y}\omega)  g(\omega)\,\dd \mathbf{P}\right) \rho(y)\,\dd y = \int_{\R^n}\left(\int_\Omega f(\omega)  g(\tau_{-\delta y}\omega)\,\dd \mathbf{P}\right) \rho(y)\,\dd y\\
&=\int_\Omega f(\omega)\left(\int_{\R^n} g(\tau_{\delta y} \rho(-y) \,\dd y\right) \,\dd \mathbf{P} = \int_\Omega f(\omega) g_\delta(\omega)\,\dd \mathbf{P}
\end{aligned}
\end{equation*}
which converges, in view of $g_\delta \to g$ in $L^1$, to $\int_\Omega fg$. Hence, $f_\delta$ converges to $f$ in the weak-$*$ topology. This completes the proof.
\end{proof}

In view of this density property, we define
\begin{equation}
\label{eq:C2U}
    \mathscr{U} = \{ \phi\in L^{\infty}(\Omega) \,:\, x\mapsto \phi(\tau_{x}\omega) \in C^{2}_{\rm b}(\mathbb{R}^n) \text{ a.s.\, for $\omega \in \Omega$}\} .
\end{equation}
By definition, for $\phi \in \mathscr{U}$, the corresponding random field $\widetilde \phi(x,\omega) := \phi(\tau_x \omega)$ is stationary and is bounded and is $C^2$ in the spatial variable, $\mathbf{P}$-almost surely. Then $\mathscr{U}$ is dense in $L^p(\Omega)$ for all $p\in [1,\infty)$. Similarly, let $\mathscr{U}^n$ consist of random vectors whose spatial translations form $C^2$ vector fields on $\R^n$.

We follow the methodology of Kosygina, Rezakhanlou and Varadhan \cite{KRV06} to define the effective Lagrangian $\overline{L}(q)$, the Legendre transform of which gives the effective Hamiltonian $\overline{H}(p)$. To do so we need to lift jump-diffusion processes in $\R^n$ to the probability space $\Omega$; this defines the so-called random walks in random environment, or the environmental process. The effective Lagrangian $\overline{L}$ is then defined via the ergodicity property of the environmental process. 

Define
\begin{equation}
\label{eq:vfB}
    \mathbf{B} =  L^{\infty}(\Omega,\mathbb{R}^n).
\end{equation}
Note that $\mathscr{U}^n$ is dense in $\mathbf{B}$ in the weak-$*$ topology. Given $b\in\mathbf{B}$. Almost surely for $\omega\in\Omega$, the realization $\widetilde b(x,\omega):= b(\tau_{x}\omega)$ is a bounded and Borel vector field. 
The martingale problem associated with \eqref{eq:controlset} is well posed and has a Feller transition density with two-sided bounds; see, e.g. \cite{ChenZhang2018}.
Therefore there exists a jump-diffusion process determined by the SDE \eqref{SDElevy} with drift $c=\widetilde{b}(x,\omega)=b(\tau_x\omega)$. Its infinitesimal generator is
\begin{equation*}
\mathcal{A}_{b}: = \mathscr{L} + \langle \widetilde b(x,\omega), \nabla \rangle.
\end{equation*}
Following \cite{KRV06} we lift this jump-diffusion process to a process $\theta^b_t$ with state space $\Omega$ by setting $\theta^b_t(\omega) = \tau_{X_t}\omega$. We use $\mathbb{Q}^{b,\omega}_0$ to denote the law of the jump-diffusion process $X_t^{b,\omega}$ starting from $0$ and with state space being \emph{C\`adl\`ag} paths in $\R^n$, and use $P^{b,\omega}$ for the law of the resulting process $\theta^b_t(\omega)$. The process is usually called the \emph{environmental process}.

Given $b \in \mathbf{B}$, $\theta^b_t$ is a Markov process with values in $\Omega$. To find its infinitesimal generator, recall that $\{\tau_{y}:\Omega\rightarrow \Omega\}_{y\in\mathbb{R}^n}$ is an ergodic group of translations on the probability space $(\Omega,\mathscr{F},\bP)$. For each $i=1,\dots,n$, let $e_i$ denote the unit vector in the $x_i$-direction, and let $\mathfrak{D}_i$ be the infinitesimal generator of the transformation group $\{\tau_{te_i}\}_{t\in\R}$. More precisely, given a real valued random variable $\phi$ on $\Omega$, let $\widetilde{\phi}(x;\omega)$ denote the random field $x\mapsto (\tau_x \phi)(\omega)$. $\mathfrak{D}_i$ is defined by
$$
(\mathfrak{D}_i \phi)(\omega) = \lim_{h\to 0} \frac{\phi(\tau_{he_i}\omega)-\phi(\omega)}{h} = \left.\frac{\partial}{\partial x_i} \widetilde{\phi}(x;\omega)
\right\rvert_{x=0}, \quad  \phi \in \Dom(\mathfrak{D}_i),
$$
where the domain $\Dom(\mathfrak{D}_i)$ consists of random variables $\phi$ so that the right-hand side above exists. Define also $\mathfrak{D} = (\mathfrak{D}_1,\dots,\mathfrak{D}_n)$. Note also $\tau_y \circ \mathfrak{D}_i = \mathfrak{D}_i\circ\tau_{y}$. Clearly, for random vector $b\in L^2(\Omega;\R^n)$ and random variable $\phi$ belonging to $\cap_i \Dom(\mathfrak{D}_i)$, we have
\begin{equation}
\label{eq:absD}
\langle b(\tau_x\omega),(\mathfrak{D}\phi)(\tau_x \omega)\rangle = \langle \widetilde{b}(x;\omega),\nabla_x\widetilde{\phi}(x;\omega)\rangle.
\end{equation}

We lift the non-local integro-differential operator $\mathscr{L}$ in \eqref{eq:diL} to an operator acting on random variables in $(\Omega,\mathscr{F},\bP)$ by setting
\begin{equation}
\label{eq:absL}
    \mathfrak{L} \phi(\omega)  = \int_{\mathbb{R}^{n}} [\phi(\tau_{y}\omega)-\phi(\omega)-\mathbf{1}_{B_1}(y)\langle \mathfrak{D} \phi(\omega), y\rangle ]\frac{K(\hat{y})}{|y|^{n+\alpha}}\,\mathrm{d}y.
\end{equation}
Equivalently, this definition is determined by
\begin{equation*}
\left(\mathfrak{L} \phi\right)(\tau_x \omega) = \mathscr{L}\widetilde{\phi}(x;\omega).
\end{equation*}

With this definition, let us check that the infinitesimal generator of $\theta^b_t$ is given by
\begin{equation}
    \label{eq:cAOmega}
\mathfrak{A}_{b} := b(\omega)\cdot\mathfrak{D} + \mathfrak{L}.
\end{equation}
Indeed, given $\phi\in \mathscr{U}$ and in view of the definition $\theta^b_t \omega = \tau_{X_t}\omega$, we see that the measure $P^{b,\omega}$ induced from $\mathbb{Q}^{b,\omega}_0$ satisfies
\begin{equation*}
    \mathbf{E}^{P^{b,\omega}} [\phi(\theta_t\omega)-\phi(\omega)] = \mathbb{E}^{\mathbb{Q}^{b,\omega}_0} [\widetilde\phi(X_t,\omega)-\widetilde\phi(0,\omega)].
\end{equation*}
Hence, we get
\begin{equation*}
\begin{aligned}
     \lim_{t\to 0}\frac{1}{t}\mathbf{E}^{P^{b,\omega}}[\phi(\theta_t\omega) - \phi(\omega)]= & \lim_{t\to 0}\frac{1}{t}\mathbb{E}^{Q^{b,\omega}_{0}}[\widetilde\phi(X_{t},\omega) - \widetilde\phi(0,\omega)] \nonumber\\
    = & \widetilde b(0,\omega)\cdot\nabla_x \widetilde\phi(0,\omega) + \mathscr{L} \widetilde\phi(x,\omega) \\
   = & b(\omega)\cdot\mathfrak{D} \phi(\omega) + \mathfrak{L} \phi(\omega),
\end{aligned}
\end{equation*}
where the relations \eqref{eq:absD} and \eqref{eq:absL} are used.

Recall that, for sufficiently smooth function $u$, we have written $\mathscr{L}u = \mathscr{J}\nabla u$, where $\mathscr{J}$ is defined for curl-free vector fields on $\R^n$ by \eqref{eq:nldivRn}. For $\R^n$-valued random vector $v$ on $\Omega$, let $\mathfrak{D}\times v$ be the tensor field $(\mathfrak{D}_j v^i - \mathfrak{D}_i v^j)$. Then assuming that $\widetilde v(\cdot,\omega) = v(\tau_{\bullet}\omega)$ is smooth in $x$, we see that $\mathfrak{D}\times v = 0$ implies that the vector field $\widetilde v(\cdot,\omega)$ is curl-free. For such random vectors, we define
\begin{equation}
\label{eq:JJ}
\mathfrak{J}v(\tau_x \omega) := \mathscr{J}\widetilde v(x,\omega).
\end{equation}
The above definition is equivalent to
\begin{equation*}
%\label{eq:nldiv}
\mathfrak{J}v(\omega) = \int_{\R^n} \left\{\left(\int_{0\to y} \langle v(\tau_{z(s)}\omega), \mathrm{d}z(s)\rangle\right) -\mathbf{1}_B (y)\langle v(\omega),y \rangle\right\}\frac{K(\hat{y})}{|y|^{n+\alpha}}\,\,\mathrm{d}y.
\end{equation*}
By those definitions, for $\phi \in \mathscr{U}$, it holds that
\begin{equation}
\label{eq:LJprob}
    \mathfrak{L} \phi = \mathfrak{J}(\mathfrak{D}\phi)
\end{equation}
For functions (and vector fields) on $\R^n$, we also defined the operator $\mathcal{R}$ in \eqref{eq:cRdef} so that 
\begin{equation*}
\mathscr{L}(u) = \nabla \cdot (\mathcal{R}\nabla u)
\end{equation*}
for smooth functions, or $\mathscr{J}(v) = \nabla \cdot (\mathcal{R}v)$ for smooth vector fields. Using the translation group $\{\tau_x\}$, we define
\begin{equation}
    \label{eq:fRdef}
\mathfrak{R} v(\omega) := \mathcal{R}\widetilde{v}(0,\omega) = \int_{\R^n} R(-y)v(\tau_y\omega)\,\dd y.
\end{equation}
We check that $\mathfrak{R}v(\tau_x \omega) = \mathcal{R}\widetilde{v}(x,\omega)$. As a consequence, for curl-free random vector $v$, we have
\begin{equation*}
    \mathfrak{J}v = \mathfrak{D}\cdot \mathfrak{R}v.
\end{equation*}
Note also that since $\mathfrak{D}$ and $\mathfrak{R}$ all commute with the translation operators $\{\tau_x\}$, for curl-free random vector $v$, we have $(\mathfrak{J}v)(\tau_x\omega) = \mathfrak{J}(\tau_x v)(\omega)$.

\medskip 

A probability density $\rho \in L^{1}(\Omega, \mathscr{F}, \bP)$ with $\rho \geq 0$ and $\int_{\Omega}\rho \mathrm{d}\bP=1$ is an invariant density with respect to the environmental process $\theta^b_t$ if for every $\phi \in \mathscr{U}$,
\begin{equation}
\label{eq:FPweak}
    \int_{\Omega} \left( b(\omega)\cdot\mathfrak{D} \phi(\omega) + \mathfrak{L}\phi(\omega)\right)\rho(\omega)\,\mathrm{d}\bP = 0.
\end{equation}
We interpret the above as the definition of $\rho$ satisfying the stationary Fokker-Planck equation
\begin{equation*}
\mathfrak{L}^*\rho=\mathfrak{D}\cdot  (b(\omega) \rho).
\end{equation*}
where $\mathfrak{L}^*$ is the dual of $\mathfrak{L}$, which, similar to \eqref{eq:absL}, is defined by
\begin{equation*}
\mathfrak{L}^* \phi(\omega)  = \int_{\mathbb{R}^{n}} [\phi(\tau_{y}\omega)-\phi(\omega)-\mathbf{1}_{B_1}(y)\langle \mathfrak{D} \phi(\omega), y\rangle ]\frac{K(-\hat{y})}{|y|^{n+\alpha}}\,\mathrm{d}y.
\end{equation*}
Under the evenness assumption on $K$, $\mathfrak{L}^*=\mathfrak{L}$. %The domains of definition of $\mathfrak{L}$ and its dual can be set as $\mathscr{U}$ defined in \eqref{eq:C2U}. 

We also denote $\mathbf{D}$ the space of probability densities $\rho:\Omega\rightarrow \mathbb{R}$ relative to $\bP$, with $\rho$ positive almost everywhere and $\rho \in \mathscr{U}$. Define 
\begin{align*}
    \mathscr{E}:=\{(b,\rho)\in\mathbf{B}\times\mathbf{D} \,:\,  \mathfrak{L}^*\rho=\mathfrak{D}\cdot  (b(\omega) \rho) \}.   
\end{align*}
Note that, given $(b,\rho) \in \mathscr{E}$, $\upd\mathbf{Q} = \rho\upd \mathbf{P}$ is the ergodic invariant measure of the environmental process $\theta_t^b$. Moreover, it is clear that the set $\mathscr{E}$ is nonempty, since for any constant vector $b$, $(b,1) \in \mathscr{E}$. Similarly, given a regular $\rho$ that is bounded from below by a positive constant, then the random vector $b:= (\mathfrak{R}^* \mathfrak{D}\rho)/\rho$ also makes $(b,\rho) \in \mathscr{E}$. 

We point out that, the importance of $\rho \mathrm{d}\bP$ being the ergodic invariant measure for the environmental process $\theta^b_t$ is, for $\bP$-a.s $\omega$ and for all integrable $F$, 
\begin{equation*}
    \lim_{t\rightarrow\infty}\frac{1}{t}\int^{t}_{0}F(\theta_s \omega)\mathrm{d}s = \int_{\Omega} F(\omega)\rho(\omega)\mathrm{d}\bP.
\end{equation*}
This motivates the definition of the effective Lagrangian given in the next subsection.

\medskip

%%%%%%%%%
\subsection{Definition of the effective Lagrangian and the effective Hamiltonian}

Following the ideas of \cite{KV-HJB}, we construct the effective Lagrangian $\overline{L}$ as
\begin{equation}
    \overline{L}(v):= \inf \{\mathbf{E}[L(b(\omega),\omega)\rho(\omega)] \;:\; {(b,\rho)\in \mathscr{E}, \mathbf{E}(b\rho)=v}\}.
\end{equation}

The effective Hamiltonian is then defined to be the Legendre transform of $\overline{L}$. This is possible because $\overline{L}$ inherits convexity from $L$ as shown in the next result.

\begin{prop}
\label{prop:propolL}
Under assumption {\upshape({\bf A})}, the effective Lagrangian $\overline L: \R^n\to \R$ defined above is a convex function and still satisfies the bounds in  {\upshape({\bf A1})}.
\end{prop}

\begin{proof} To check the convexity of $\overline{L}(v)$ in $v$, fix any $v_1,v_2\in \R^n$ and $\lambda\in (0,1)$. For any $r>0$ and each $i\in \{1,2\}$, by the definitions of $\ol{L}(v_i)$, we can find $(b_i,\rho_i) \in \mathscr{E}$ with $\mathbf{E}(b_i\rho_i) = v_i$, such that
\begin{equation*}
    \overline{L}(v_i) + r\ge \mathbf{E}[L(b_i(\omega),\omega)\rho_i], \quad i=1,2.
\end{equation*}
Consider the pair $(b,\rho)$ defined by
\begin{equation*}
b=\frac{\lambda b_1\rho_1 + (1-\lambda) b_2\rho_2}{\lambda \rho_1 + (1-\lambda)\rho_2}, \quad \rho = \lambda \rho_1 + (1-\lambda) \rho_2.
\end{equation*}
They clearly satisfy $\mathbf{E}(b\rho) = \lambda v_1 + (1-\lambda) v_2$. It is also straightforward to check that \eqref{eq:FPweak} holds for this pair $(b,\rho)$. Therefore, $(b,\rho) \in \mathscr{E}$.

By the convexity of $v\mapsto L(v,\omega)$, and for the pair $(b,\rho)$ given above, we have
\begin{equation*}
    \begin{aligned}
    \mathbf{E}[L(b,\omega)\rho] &\le \mathbf{E}\left[\rho\left(\frac{\lambda \rho_1}{\lambda \rho_1 + (1-\lambda)\rho_2}L(b_1,\omega) + \frac{(1-\lambda) \rho_2}{\lambda \rho_1 + (1-\lambda)\rho_2}L(b_2,\omega)\right)\right]\\
    &=\mathbf{E}[\lambda L(b_1,\omega)\rho_1 + (1-\lambda) L(b_2,\omega)\rho_2]\\
    &\le \lambda \overline{L}(v_1) + (1-\lambda) \overline{L}(v_2) + r.
    \end{aligned}
\end{equation*}
This implies that
\begin{equation*}
    \ol{L}(v) \le \lambda \ol{L}(v_1) + (1-\lambda)\ol{L}(v_2) + r.
\end{equation*}
Sending $r\to 0$, we obtain the convexity of $\overline{L}$.

To check the growth rate of $\ol{L}$, fix any $v\in \R^n$, we see that the pair $(b,\rho)$ given by $b(\omega)\equiv v$ and $\rho(\omega)\equiv 1$ belongs to $\mathscr{E}$ and satisfies $\mathbf{E}[b\rho] = v$. In view of the upper bound in \eqref{A2L} and the definition of $\ol{L}$ we have $\overline{L}(v)\le |v|^{\beta'} + c_2$. This yields the upper bound of $\ol{L}(v)$. To prove the lower bound, we use the lower bound of $L$ and then Jensen's inequality to obtain, for any $(b,\rho)\in \mathscr{E}$ with $\mathbf{E}[b\rho]=q$,
\begin{equation*}
\begin{aligned}
\mathbf{E}[L(b,\omega)\rho] \ge \int_\Omega \left(c_1'|b(\omega)|^{\beta'}-c_4\right)\,\rho \mathrm{d}\mathbf{P} \ge c'_1\left|\int_\Omega b(\omega)\,\rho \mathrm{d}\mathbf{P}\right|^{\beta'} -c_4
\end{aligned}
\end{equation*}
This yields $\overline{L}(q)\ge c_1'|q|^{\beta'}-c_4$, i.e., the lower bound in \eqref{A2L}, and completes the proof.
\end{proof}

The effective Hamiltonian $\overline{H}$ is hence defined by
\begin{equation}\label{EffH}
    \overline{H}(p):= \sup_{q\in\mathbb{R}^n}\Big( \langle p, q \rangle - \overline{L}(q) \Big) = \sup_{(b,\rho)\in \mathscr{E}}\Big( \langle p, \mathbf{E}[b(\omega)\rho(\omega)] \rangle - \mathbf{E}[L(b(\omega),\omega)\rho(\omega)]\Big).
\end{equation}
The properties about $\overline{L}$ in Proposition \ref{prop:propolL} immediately transfer to the effective Hamiltonian: $\overline{H}(p)$ is convex in $p$ and $\overline{H}$ satisfies \eqref{A1H}. 

With the above definition of effective Hamiltonian, the proclaimed effective Hamilton-Jacobi equation \eqref{hHJB} is determined. Since $\overline{H}$ is convex, the Hopf-Lax formula yields the following representation formula for its unique viscosity solution:
\begin{equation}
\label{eq:HopfLax}
    \overline{u}(t,x) = \sup_{y\in\mathbb{R}^n}\left(u_{0}(y)- t \overline{L}(\frac{y-x}{t})\right).
\end{equation}
Combine this formula with the expression of $\overline{L}$, we get, for all $x\in \R^n$ and $t\ge 0$,
\begin{equation}
    \label{eq:ubarformula}
\overline{u}(t,x) = \sup_{(b,\rho) \in \mathscr{E}} \Big( u_0(x+t\mathbf{E}[b(\omega)\rho(\omega)]) - t\mathbf{E}[L(b(\omega),\omega)\rho(\omega)]\Big).
\end{equation}
We will use the following simple rescaling invariance
\begin{equation*}
    \overline{u}(t,x) = \eps \overline{v}(t/\eps,x/\eps)%\eps \overline{S}(t/\eps)(\eps^{-1}u_0(\eps\,\cdot))(x/\eps)
\end{equation*}
where $\overline{v}$ is the solution to the HJ equation \eqref{hHJB} with initial condition $v(0,x) = \eps^{-1}u_0(\eps x)$.

%%%%%%%%%%%
%%%%%%%%%%%
\section{Regularity and Lower Bounds}
\label{sec:lowerbdd}

\subsection{Some regularity results} Using the representation formula for the solution $u_\eps$ to the HJB equation \eqref{HJB}, we first establish some regularity results for $u_\eps$. They play an important role in our analysis. 
For those results, we impose the stronger Lipschitz condition \eqref{BLip} on the initial data.

\medskip

The main result of this subsection is the following theorem which exhibits regularity results for $u_\eps$ with explicit dependence on $\eps$. The proof can be found at the end of this subsection.

\begin{thm}\label{thm:regularity} Suppose Assumptions {\upshape ({\bf A1})} and {\upshape ({\bf A3$'$})} hold, and $\alpha \in (1,2)$ in the HJB equation \eqref{HJB}. Then there is a universal constant $C<\infty$ depending only on the bounding parameters in {\upshape({\bf A1})} and {\upshape({\bf A3$'$})} such that, for all $t>0$, $x,y\in \R^n$ and $\omega\in \Omega$,
\begin{equation}
    \label{eq:regx}
|u_\eps(t,x,\omega)-u_\eps(t,y,\omega)| \le C\left(|x-y| +\eps^{1-\frac1\alpha} |x-y|^{\frac1\alpha} + \eps^{(1-\frac1\alpha)\frac{1}{\beta'}} |x-y|^{1-(1-\frac1\alpha)\frac{1}{\beta'}}\right).
\end{equation}
Moreover, for any $T<\infty$, there is a constant $C = C_T$ depending on the bounding parameters in {\upshape({\bf A1})} and {\upshape({\bf A3$'$})} and on $T$ such that for all $t,s \in [0,T]$, $x\in \R^n$ and $\omega \in \Omega$,
\begin{equation}
    \label{eq:regt}
    |u_\eps(t,x,\omega)-u_\eps(s,x,\omega)| \le C\left( |t-s| + \eps^{1-\frac1\alpha}|t-s|^{\frac1\alpha}\right) + C_T|t-s|^{\frac1\beta}.
\end{equation}
\end{thm}

Such regularity results are clearly important. In particular, together with the result in Lemma \ref{CuI} below which yields locally (in $t$) uniform (in $x$) bound for $|u_\eps(t,x)|$, the estimates \eqref{eq:regx} and \eqref{eq:regt} imply that for each fixed $T,R>0$, the family $\{u_\eps(\cdot,\cdot,\omega) \,:\, \omega\in \Omega, \eps\in (0,1)\}$ is a uniformly bounded and equi-continuous family.

\medskip

As is standard (see, e.g.,\cite{KRV06}), we first show that the control set $\mathscr{C}$ of drift fields can be reduced to contain those with better behaviors.

\begin{lem}\label{Cset}
Under the same assumptions of Theorem \ref{thm:regularity}. Consider the control formula \eqref{vf} for the solution. Then the control set $\mathscr{C}$ in \eqref{vf} can be replaced by the subset $\mathscr{C}^{\ast} \subset \mathscr{C}$ defined by: 
$c\in \mathscr{C}^*$ if and only if there exists a $C < \infty$ depending only on the bounding parameters in {\upshape({\bf A1})} and {\upshape ({\bf A3$'$})} so that for all $t\ge 0$ and $x\in \R^n$,
\begin{equation}\label{sCset}
c'_1 \mathbb{E}^{Q^{c,\epsilon}_{x}} \int_0^t|c(\eps^{-1}s,\eps^{-1}X^\eps_s)|^{\beta'}\,\mathrm{d}s  \leq C(t + \epsilon^{1-\frac{1}{\alpha}} t^{\frac{1}{\alpha}}).
\end{equation}
%where $\xi_\epsilon$ is defined in \eqref{eq:xiepsdef}.
\end{lem}
\begin{proof}
In the optimal control formula \eqref{vf}, set the drift field $c$ to be zero, and let $Q^{0,\eps}_x$ be the jump diffusion process given by \eqref{eq:Xepsc} with $c=0$. By Assumptions {\upshape ({\bf A1})} and {\upshape ({\bf A3$'$})}, we find 
\begin{equation}
\label{eq:lowercontrol}
\begin{aligned}
    u_{\epsilon}(t,x,\omega) - u_{0}(x) \geq &  \mathbb{E}^{Q^{0,\epsilon}_{x}} \left( u_{0}(X^{\epsilon}_{t}) - u_{0}(x) - \xi_{\epsilon}(t) \right) \nonumber \\
    \geq &  \mathbb{E}^{Q^{0,\epsilon}_{x}}\left(-K|X^{\epsilon}_{t} - x| - \xi_{\epsilon}(t) \right) \nonumber \\
    \geq & - K\mathbb{E}^{Q^{0,\epsilon}_{x}} | \epsilon L^{\alpha}_{t/\epsilon}| - c_2 t \nonumber \\
    \geq & - K\epsilon^{1-\frac{1}{\alpha}} \mathbb{E}^{Q^{0,1}_{x}} |L^{\alpha}_{t}| - c_2 t\nonumber \\
    \geq & - CK \epsilon^{1-\frac{1}{\alpha}} t^{\frac{1}{\alpha}} - c_2 t.
\end{aligned}
\end{equation}
In the last two lines we used the scaling property \eqref{eq:Ltass} and standard bound for $L^\alpha_t$. The bound above provides a lower bound of the error $u_\eps(t,x,\omega) -u_0(x)$, and therefore in the control representation we need to consider controls (drifts) $c(t,x)$'s which yield relative gain that beats this lower bound. In particular, by \eqref{vf} and \eqref{BLip}, such $c$ should verify
\begin{equation}\label{LBE1}
    \mathbb{E}^{Q^{c,\epsilon}_{x}} [K|X^{\epsilon}_{t} - x| - \xi_{\epsilon}(t)] \geq  - CK\epsilon^{1-\frac{1}{\alpha}} t^{\frac{1}{\alpha}} - c_2 t.   
\end{equation}
On the other hand, for the SDE \eqref{eq:Xepsc}, using the aforementioned bound of $\eps L^\alpha_{t/\eps}$, the estimates in \eqref{A2L} and H\"{o}lder's inequality, we have
\begin{align}\label{LBE2}
    \mathbb{E}^{Q^{c,\epsilon}_{x}} |X^{\epsilon}_{t} - x| \leq & \mathbb{E}^{Q^{c,\epsilon}_{x}}\left( \int^{t}_{0} |c(\epsilon^{-1}s, \epsilon^{-1}X^{\epsilon}_{s})| \mathrm{d}s \right) + C\epsilon^{1-\frac{1}{\alpha}} t^{\frac{1}{\alpha}} \nonumber \\
    \leq & t^{1/\beta} \mathbb{E}^{Q^{c,\epsilon}_{x}}\left( \int^{t}_{0} |c(\epsilon^{-1}s, \epsilon^{-1}X^{\epsilon}_{s})|^{\beta'} \mathrm{d}s \right)^{1/\beta'}  + C\epsilon^{1-\frac{1}{\alpha}} t^{\frac{1}{\alpha}} \nonumber \\
    \leq & \frac{t^{1/\beta}}{{c'_1}^{1/\beta'}} \left(\mathbb{E}^{Q^{c,\epsilon}_{x}}\xi_{\epsilon}(t) + c_4 t \right)^{1/\beta'}  + C\epsilon^{1-\frac{1}{\alpha}} t^{\frac{1}{\alpha}}.
\end{align}
Plugging above estimate (\ref{LBE2}) into (\ref{LBE1}), the control $c\in\mathscr{C}$ of interest should satisfy
\begin{equation}
    K t^{\frac{1}{\beta}} {c'_1}^{-\frac{1}{\beta'}}(\mathbb{E}^{Q^{c,\epsilon}_{x}}\xi_{\epsilon}(t) + c_4 t) ^{\frac{1}{\beta'}}  - (\mathbb{E}^{Q^{c,\epsilon}_{x}}\xi_{\epsilon}(t) + c_4 t) + 2C K\epsilon^{1-\frac{1}{\alpha}} t^{\frac{1}{\alpha}} + (c_2+c_4)t  \geq 0.
\end{equation}
For the quantity $\Theta = \mathbb{E}^{Q^{c,\eps}_x} \xi_\eps(t) + c_4t$, the above inequality is of the form $A\Theta^{1/\beta'} - \Theta + B\ge 0$ for some positive quantities $A,B$ that can be easily read from the inequality above. Applying the Young's inequality we get
\begin{equation*}
    \Theta \le A^{\beta} + \beta B.
\end{equation*}
From this we deduce that we only need to consider $c\in \mathscr{C}$ such that
\begin{equation*}
c'_1 \mathbb{E}^{Q^{c,\eps}_x}\int_0^t |c(\eps^{-1}s,\eps^{-1}X^\eps_s)|^{\beta'}\,\mathrm{d}s \le \Theta \le A^\beta + \beta B.
\end{equation*}
Plug the explicit forms of $A,B$ into the above inequality and observe that all the inequalities above are uniform in $t,x$ and $\omega$. We hence have found the desired reduced control set $\mathscr{C}^*$.
\end{proof}

The next lemma shows the continuity of $u_{\epsilon}$ with respect to the initial data, and it also establishes a bound of $\sup_{t\in [0,T]} \sup_{x\in \R^n}|u_\eps(t,x)|$ that is uniform in $\eps \in (0,1)$, for each fixed $T>0$.

\begin{lem}\label{CuI}
Under the same assumptions of Theorem \ref{thm:regularity},  
there is a constant $C<\infty$ depending only on the bounding parameters in {\upshape({\bf A1})} and {\upshape({\bf A3$'$})} such that for every $\epsilon \in (0,1]$,
\begin{equation}\label{Eu1}
    \sup_{\omega,x} |u_{\epsilon}(t,x,\omega) - u_{0}(x)| \leq C\left(t + \eps^{1-\frac1\alpha} t^{\frac1\alpha} + \eps^{(1-\frac1\alpha)\frac{1}{\beta'}} t^{1-(1-\frac1\alpha)\frac{1}{\beta'}}  \right).
\end{equation}
\end{lem}
\begin{proof}
Lemma \ref{Cset} implies that $u_{\epsilon}$ satisfies 
\begin{equation}
    u_{\epsilon}(t,x,\omega) - u_{0}(x) = \sup_{c\in\mathscr{C}^{\ast}} \mathbb{E}^{Q^{c,\epsilon}_{x}}[u_{0}(X^\eps_t)
    - u_{0}(x)-\xi_{\epsilon}(t)].
\end{equation}
By \eqref{BLip} and the simple fact that $-\xi_\eps(t) \le c_4 t$, we have%, (\ref{LBE2}), and Lemma \ref{Cset}, we have for every $\delta>0$,
\begin{align*}
    u_{\epsilon}(t,x,\omega) - u_{0}(x) &\le \sup_{c\in\mathscr{C}^{\ast}} \mathbb{E}^{Q^{c,\epsilon}_{x}}[K |X^{\epsilon}_t - x|] + c_4 t\\
    &\leq \sup_{c\in\mathscr{C}^{\ast}}\mathbb{E}^{Q^{c,\epsilon}_{x}}\left( K\int^{t}_{0} |c(\epsilon^{-1}s, \epsilon^{-1}X^{\epsilon}_{s})| \mathrm{d}s \right) + C K \epsilon^{1-\frac{1}{\alpha}} t^{\frac{1}{\alpha}} +c_4 t.
\end{align*}
Combining the above with the fact that $c \in \mathscr{C}^*$ guarantees \eqref{Cset}, we obtain an upper bound of (\ref{Eu1}). The inequality obtained in \eqref{eq:lowercontrol} gives a lower bound. The proof is complete.
\end{proof}

%%%%%%%%%
\medskip

Finally we can complete the proof of Theorem \ref{thm:regularity}.
\begin{proof}[Proof of Theorem \ref{thm:regularity}] \emph{Regularity in the spatial variable}. Fix two different points $x,y\in \R^n$ and $t>0$. Suppose $c \in \mathscr{C}$ is an optimal control for $u_\eps(t,x)$, i.e.,
\begin{equation*}
u_\eps(t,x) = \mathbb{E}^{\mathbb{Q}^{c}_{x/\eps}} \left( u_0(\eps X_{t/\eps}) - \eps \int_0^{t/\eps} L(c(s,X_s),\tau_{X_s}\omega)\,\mathrm{d}s\right).
\end{equation*}
We consider the coupled path $Y_s$ defined by
\begin{equation*}
Y_s = \left\{
\begin{aligned}
&X_s + \frac{y-x}{\eps} - \frac{s(y-x)}{|y-x|}, \quad &0\le s\le \eps^{-1}|y-x|,\\
&X_s, \quad &s\ge \eps^{-1}|y-x|.
\end{aligned}
\right.
\end{equation*}
Then $Y_s$ is a jump-diffusion process starting from $y/\eps$ at time zero. Clearly we have
\begin{equation}
\label{eq:XYcouple-1}
\eps |Y_{t/\eps} - X_{t/\eps}| \le |y-x|, \quad \forall t\ge 0.
\end{equation}
The generator of $Y_s$ is $\mathscr{L} + \langle c_1(s,z),\nabla\rangle$ where the drift $c_1$ has the form
\begin{equation*}
c_1(s,z) = \left\{
\begin{aligned}
&c\left(s,z - \frac{y-x}{\eps} + \frac{s(y-x)}{|y-x|}\right) - \frac{y-x}{|y-x|}, \quad &0\le s\le \eps^{-1}|y-x|,\\
&c(s,z), \quad &s\ge \eps^{-1}|y-x|.
\end{aligned}
\right.
\end{equation*}
We check that $c_1\in \mathscr{C}$ and $|c-c_1|\le 1$ uniformly. For $t\ge |y-x|$, we have $X_{t/\eps} = Y_{t/\eps}$. Using \eqref{A2L} we also have
\begin{equation*}
\begin{aligned}
&\eps\int_0^{t/\eps} L(c_1(s,Y_s),\tau_{Y_s}\omega)\,\mathrm{d}s - \eps\int_0^{t/\eps} L(c_1(s,X_s),\tau_{X_s}\omega) \,\mathrm{d}s \\
= &\eps\int_0^{|y-x|/\eps} L(c_1(s,Y_s),\tau_{Y_s}\omega)-  L(c(s,X_s),\tau_{X_s}\omega)\,\mathrm{d}s\\
\ge & -(c_2+c_4)|y-x| - c'_3 \eps \int_0^{|y-x|/\eps} |c(s,X_s)|^{\beta'}\,\mathrm{d}s.
\end{aligned}
\end{equation*}
In view of the variational formula \eqref{vf} and by Lemma \ref{Cset}, we get
\begin{equation*}
u_\eps(t,y)-u_\eps(t,x) \ge -(c_2+c_4)|y-x| - C(|y-x|+\eps^{1-\frac1\alpha}|y-x|^{\frac1\alpha}).
\end{equation*}
For $t< |y-x|$, by Lemma \ref{CuI} we have
\begin{equation*}
    u_\eps(t,y)-u_\eps(t,x) \ge u_0(y) - u_0(x) - 2C\left(t + \eps^{1-\frac1\alpha} t^{\frac1\alpha} + \eps^{(1-\frac1\alpha)\frac{1}{\beta'}} t^{1-(1-\frac1\alpha)\frac{1}{\beta'}}\right).
\end{equation*}
Using \eqref{BLip} and $t< |x-y|$ we obtain
\begin{equation*}
    u_\eps(t,y)-u_\eps(t,x) \ge -C\left(|x-y| +\eps^{1-\frac1\alpha} |x-y|^{\frac1\alpha} + \eps^{(1-\frac1\alpha)\frac{1}{\beta'}} |x-y|^{1-(1-\frac1\alpha)\frac{1}{\beta'}}\right).
\end{equation*}
The roles of $x$ and $y$ can be exchanged, and hence we have proved \eqref{eq:regx}. Note that this bound is uniform for $\omega \in \Omega$ and $t>0$.

\medskip

\emph{Regularity in the time variable}. Here we establish some regularity of $u_\eps$ (uniform in $\omega$ and dependence on $\eps$ is shown explicitly and can be made uniform in $\eps$ as well) in time $t$. Fix $x\in \R^n$, and fix any $t>0$ and $0< t_1<t$. Suppose $c\in \mathscr{C}^*$ is an optimal velocity field for $u_\eps(t,x)$. By the control representation, we have
\begin{equation*}
\begin{aligned}
    u_\eps(t,x)-u_\eps(t_1,x) &\le \mathbb{E}^{Q^c_{x/\eps}}\left[u_0(\eps X_{t/\eps}) - u_0(\eps X_{t_1/\eps}) - \eps \int_{t_1/\eps}^{t/\eps} L(c(X_s),\tau_{X_s}
    \omega)\,\mathrm{d}s\right]\\
    &\le \mathbb{E}^{Q^c_{x/\eps}}\left[ K\eps |X_{t/\eps}-X_{t_1/\eps}|\right] + c_4(t-t_1),
    \end{aligned}
\end{equation*}
where we also used \eqref{BLip} and \eqref{A2L}. By the equation \eqref{SDElevy} satisfied by $X_\bullet$, we have
\begin{equation*}
    \eps \mathbb{E}^{Q^c_{x/\eps}} |X_{t/\eps}-X_{t_1/\eps}| \le \mathbb{E}^{Q^c_{x/\eps}} \eps \int_{t_1/\eps}^{t/\eps} |c(s,X_s)|\,\mathrm{d}s + C\eps^{1-\frac1\alpha}(t-t_1). 
\end{equation*}
Since we may assume that $c\in \mathscr{C}^*$, by \eqref{sCset} we deduce that
\begin{equation*}
    u_\eps(t,x)-u_\eps(t_1,x) \le C(t-t_1)^{\frac1\beta}(t + \epsilon^{1-\frac{1}{\alpha}} t^{\frac{1}{\alpha}})^{\frac{1}{\beta'}} + C\eps^{1-\frac1\alpha}(t-t_1).
\end{equation*}

Note that by the dynamic programming principle (DPP),
\begin{equation*}
    u_\eps(t,x,\omega) = \sup_{c\in \mathscr{C}} \mathbb{E}^{Q^c_{x/\eps}} \left[ u_\eps(t_1,\eps X_{(t-t_1)/\eps}) - \eps\int_0^{(t-t_1)/\eps} L(c(X_s),\tau_{X_s}\omega)\,\mathrm{d}s\right].
\end{equation*}
Suppose $c_1\in \mathscr{C}$ is an optimal velocity field for $u_\eps(t_1,x)$, we may extend to some $\hat{c}_1\in \mathscr{C}$ such that $\hat c_1 = c_1$ for $0\le s\le t_1$. By the DPP above, we get
\begin{equation*}
    u_\eps(t,x)-u_\eps(t_1,x) \ge \mathbb{E}^{Q^{\hat{c}_1}_{x/\eps}} \left[u_\eps(t_1,\eps X_{(t-t_1)/\eps}) - u_\eps(t_1,x) - \eps\int_0^{(t-t_1)/\eps} L(\hat{c}_1(X_s),\tau_{X_s}\omega)\,\mathrm{d}s\right].
    \end{equation*}
Using the regularity of $u_\eps$ in the space variable and setting $h=t-t_1$, we can bound the right hand side above from below by
    \begin{equation*}
    -c_2(t-t_1) - C\mathbb{E}^{Q^{\hat{c}_1}_{x/\eps}}\left(|\eps X_{h/\eps}-x| +\eps^{1-\frac1\alpha} |\eps X_{h/\eps}-x|^{\frac1\alpha} + \eps^{(1-\frac1\alpha)\frac{1}{\beta'}} |\eps X_{h/\eps}-x|^{1-(1-\frac1\alpha)\frac{1}{\beta'}}\right).
\end{equation*}
Using the control of $\eps X_\bullet -x$ as before, and \eqref{sCset} again, the above is further bounded
\begin{equation*}
-C(t-t_1)-C\eps^{1-\frac1\alpha}(t-t_1)^{\frac1\alpha}
\end{equation*}
for some universal constant $C<\infty$.
\end{proof}
%%%%%%%%

\subsection{Lower bound of the homogenization error}

In view of the representation formulas, namely, \eqref{vf} for $u_\eps$ and \eqref{eq:ubarformula} for $\overline{u}$, we plan to use the ergodic theorem as the starting point to find a lower bound of $u_\eps-\overline{u}$. %From the representation formula \eqref{eq:HopfLax}, for any $t>0$, we have
%\begin{equation}
%\label{eq:ubarcontrol}
%\begin{aligned}
%    \overline{u}(t,0) &= \sup_{y\in \R^n} \left(u_0(y)-\inf_{(b,\rho)\in \mathscr{E}\,:\,\mathbf{E}(b(\omega)\rho)(\omega))=y/t} \,t\mathbf{E}[L(b(\omega),\omega)\rho(\omega)]\right)\\
%    &= \sup_{y\in \R^n} \sup_{(b,\rho)\in \mathscr{E}\,:\,\mathbf{E}(b(\omega)\rho(\omega))=y/t} \left(u_0(y)-t\mathbf{E}[L(b(\omega),\omega)\rho(\omega)]\right)\\
%&=\sup_{(b,\rho)\in\mathscr{E}} \left(u_0(t\mathbf{E}[b(\omega)\rho(\omega)] - t\mathbf{E}[L(b(\omega),\omega)\rho(\omega)]\right)
%   \end{aligned}
%\end{equation}

We have seen that if $(b,\rho) \in \mathscr{E}$, then $\rho\,\mathrm{d}\mathbf{P}$ is an invariant ergodic measure on $\Omega$ for the environmental process $\theta_t$ (more precisely, for the corresponding Markov semigroup with generator $\mathfrak{A}_b$). The ergodic theorem then yields
%the For every $(b,\rho) \in \mathscr{E}$, we set
%\begin{equation*}
%    m(b,\rho) = \int b(\omega)\rho(\omega)\mathrm{d}\mathbf{P},
%\end{equation*}
%\begin{equation*}
%    h(b,\rho) = \int L(b(\omega),\omega)\rho(\omega)\mathrm{d}\mathbf{P}.
%\end{equation*}
%Then the ergodic result (Lemma \ref{ErgodicL}) yields that for almost all $\omega\in\Omega$ with respect to $\bP$ and the diffusion process $X_{t}$ with induced measure $Q^{c}_{x}$, we have
\begin{equation*}
%\label{eq:Qcergodic}
\begin{aligned}
    &\lim_{\epsilon\rightarrow 0} \epsilon \int^{t/\epsilon}_{0} b(\tau_{X_{s}}\omega) \mathrm{d}s = t\int_\Omega b(\omega)\rho(\omega)\mathrm{d}\bP = t\mathbf{E}[b(\omega)\rho(\omega)],\\% = m(b,\rho)t,
%\end{equation*}
%\begin{equation*}
    &\lim_{\epsilon\rightarrow 0} \epsilon \int^{t/\epsilon}_{0} L(b(\tau_{X_{s}}\omega),\omega(s))\mathrm{d}s = t\int_\Omega L(b(\omega),\omega)\rho(\omega)\mathrm{d}\bP = t\mathbf{E}[L(b(\omega),\omega)\rho(\omega)].%= h(b,\rho)t,
    \end{aligned}
\end{equation*}
The above holds in $L^{1}(Q^{b,\omega}_{0})$ and $Q^{b,\omega}_0$ almost surely in the path space, and $\bP$-a.s.\,for $\omega \in \Omega$. Let $c(x,\omega)=b(\tau_x,\omega)$. %Using the boundedness of $b$ and by the dominated convergence theorem, the above also yields, 
We also have, $\bP$-a.s.\,in $\Omega$,
\begin{equation}
    \label{eq:Qcergodic-1}
\begin{aligned}
     &\lim_{\epsilon\rightarrow 0}\, \mathbb{E}^{Q^{c}_{0}}\left(\left|\epsilon\int^{t/\epsilon}_{0}b(\tau_{X_{s}}\omega)\mathrm{d}s - t\mathbf{E}[b(\omega)\rho(\omega)] \right| \right)= 0,\\
    &\lim_{\epsilon\rightarrow 0}\, \mathbb{E}^{Q^{c}_{0}}\left(\left|\epsilon\int^{t/\epsilon}_{0}L(b(\tau_{X_{s}}\omega),\tau_{X_{s}}\omega)\mathrm{d}s - t\mathbf{E}[L(b(\omega),\omega)\rho(\omega)] \right| \right) =  0.
\end{aligned}
\end{equation}
Moreover, all of those convergence results are uniform in $t$ in any finite time interval $[0,T]$.
By Egoroff's theorem, for every $\eta >0$, there exists $N_{\eta} \in\mathscr{F}$ with $\bP(N_{\eta})\geq 1-\eta$, such that
\begin{align}\label{ELE1}
     \lim_{\epsilon\rightarrow 0}\sup_{\omega\in N_{\eta}}\sup_{0\leq t \leq T}\mathbb{E}^{Q^{c}_{0}}\left(\left|\epsilon\int^{t/\epsilon}_{0}b(\tau_{X_{s}}\omega)\mathrm{d}s - t\mathbf{E}[b(\omega)\rho(\omega)] \right| \right)= 0,
\end{align}
and 
\begin{align}\label{ELE2}
    \lim_{\epsilon\rightarrow 0}\sup_{\omega\in N_{\eta}}\sup_{0\leq t \leq T}\mathbb{E}^{Q^{c}_{0}}\left(\left|\epsilon\int^{t/\epsilon}_{0}L(b(\tau_{X_{s}}\omega),\tau_{X_{s}}\omega)\mathrm{d}s - t\mathbf{E}[L(b(\omega),\omega)\rho(\omega)]\right| \right) =  0.
\end{align}

\begin{lem}\label{Con3}
Suppose that Assumptions {\upshape({\bf A1})} and {\upshape({\bf A3$'$})} hold. Then there is an event $\Omega_0 \subset \Omega$ with $\mathbf{P}(\Omega_0) = 1$ such that, for all $T>0$,
\begin{equation}\label{Eu00}
    \liminf_{\epsilon\rightarrow 0} %\inf_{\omega\in N_{\eta}}
    \inf_{0\leq t \leq T} \,[u_{\epsilon}(t,0,\omega) - \overline{u}(t,0)] \ge 0, \qquad \forall \omega \in \Omega_0.
\end{equation}
\end{lem}

\begin{proof} In view of the formula \eqref{eq:ubarformula} with $x=0$, for any $\eta >0$, there is pair $(b,\rho) \in \mathscr{E}$ such that
\begin{equation}
\label{eq:uepst0-1}
u_0(t\mathbf{E}[b(\omega)\rho(\omega)]) - t\mathbf{E}[L(b(\omega),\omega)\rho(\omega)] \ge \overline{u}(t,0) -\eta. 
\end{equation}
For each $\omega \in \Omega$, consider the time-homogeneous drift field $c(x,\omega) = b(\tau_x \omega)$. Then $c\in \mathscr{C}$, and by the variational formula (\ref{vf}), we have
\begin{equation*}
    u_{\epsilon}(t,0,\omega) \geq \mathbb{E}^{Q^{c}_{0}}\left( u_{0}(\epsilon X_{t/\epsilon}) - \epsilon\int^{t/\epsilon}_0 L(c(s,X_s),\tau_{X_{s}}\omega)\mathrm{d}s. \right)
\end{equation*}
Combining those inequalities above, we get 
\begin{equation}
\label{eq:errort0-1}
\begin{aligned}
     u_\eps(t,0,\omega)-\overline{u}(t,0) &\ge \mathbb{E}^{Q^c_0}\left(t\mathbf{E}[L(b(\omega),\omega)\rho(\omega)] - \eps\int_0^{t/\eps} L(b(\tau_{X_s}\omega),\tau_{X_s}\omega)\,\mathrm{d}s\right)\\
         &\quad
      + \mathbb{E}^{Q^c_0} \left(u_0(\eps X_{t/\eps}) - u_0(t\mathbf{E}[b\rho]\right) - \eta.
     \end{aligned}
\end{equation}
By \eqref{eq:Qcergodic-1} the first item above vanishes as $\eps\to 0$, $\mathbf{P}$-a.s., and uniformly for $t\in [0,T]$. For the second term, we note that the jump diffusion process $X_s$ in the above expression satisfies
\begin{equation*}
    \epsilon X_{t/\epsilon} - t\mathbf{E}[b\rho] = \left(\eps\int^{t/\eps}_{0}b(\tau_{X_s}\omega)\mathrm{d}s - t\mathbf{E}[b\rho]\right) + \eps L^{\alpha}_{t/\eps}.
\end{equation*}
In view of \eqref{eq:Ltass} and the fact that $\alpha >1$, the diffusion term $\epsilon L^{\alpha}_{t/\epsilon}$ vanishes to zero as $\eps\to 0$. Using the Lipschitz bound \eqref{BLip} of $u_0$, and \eqref{eq:Qcergodic-1} again, we see that the second term in \eqref{eq:uepst0-1} vanishes also as $\eps\to 0$, $\mathbf{P}$-a.s., and uniformly for $t\in [0,T]$. This completes the proof.
\end{proof}

Next, we need to transfer the lower bound for the homogenization error at $x=0$ to a lower bound that is locally uniform for all spatial points. There is a standard trick usually attributed to Varadhan, which we carry out carefully below. 

First, we use the translation property to transfer the $\mathbf{P}$-a.s.\,point-wise lower bound from $x=0$ to any point $x\in \R^n$.

\begin{cor}\label{cor:pwlbdd}
Suppose that Assumptions {\upshape({\bf A1})} and {\upshape({\bf A3$'$})} hold. Let $T>0$. Then for each fixed $x\in \R^n$, we can find a set $\Omega_x\in \mathscr{F}$ with $\mathbf{P}(\Omega_x) = 1$, so that
\begin{equation}\label{lem:lbddx}
    \liminf_{\epsilon\rightarrow 0} %\inf_{\omega\in N_{\eta}}
    \inf_{0\leq t \leq T} \,[u_{\epsilon}(t,x,\omega) - \overline{u}(t,x)] \ge 0, \qquad \forall \omega \in \Omega_x.
\end{equation}
\end{cor}
\begin{proof} We have seen that for fixed $x\in \R^n$, $u_\eps(t,x,\omega) = u_\eps^x(t,0,\tau_x \omega)$. Here, $u_\eps^x(t,\cdot) = S_\eps(t;\tau_x\omega) T_x u_0$, that is the solution to the oscillatory HJB equation \eqref{HJB} with environment $\tau_x \omega$ and with initial data $T_x u_0 = u_0(\cdot + x)$. We have also seen that $\overline{u}(t,x) = \overline{u}^x(t,0)$, where $\overline{u}^x(t,\cdot) = \overline{S}(t)T_x u_0$ is the solution to the HJ equation \eqref{hHJB} with initial data $T_x u_0$. Since $T_xu_0$ enjoys the same Lipschitz bound condition ({\bf A3$'$}) as $u_0$, and because $\tau_x$ preserves $\mathbf{P}$ on $\Omega$, the same proof of Lemma \ref{Con3} works for $u_\eps^x$ and $\overline{u}^x$. The desired result follows. 
\end{proof}

%To complete the argument we use the oscillation control of $u_\eps$, and the ergodicity of $\{\tau_y\}_{y\in \R^n}$.

\begin{thm}\label{LBR}
Assume that assumptions {\upshape ({\bf A1})}, {\upshape ({\bf A2})}, and {\upshape({\bf A3$'$})} hold. Let $u_{\epsilon}(t,x,\omega)$ be the solution of the {\upshape HJB}  equation (\ref{HJB}), and $\overline{u}(t,x)$ be the solution of the homogenized {\upshape HJ} equation (\ref{hHJB}). Then for every $R,T>0$, we can find $\widetilde{\Omega} \subset \Omega$ with $\mathbf{P}(\widetilde\Omega) = 1$, such that for all $\omega \in \widetilde\Omega$, 
\begin{equation}\label{eq:uflbdd}
    \liminf_{\epsilon\rightarrow 0}\inf_{t\in[0,T]}\inf_{|x|\leq R}(u_{\epsilon}(t,x,\omega) - \overline{u}(t,x)) \geq 0 .
\end{equation}
\end{thm}

\begin{proof} First, apply Corollary \ref{cor:pwlbdd} to each $x\in \mathbb{Q}^n$ (that is, all $n$ coordinates of $x$ belong to $\mathbb{Q}$), and let $\widetilde\Omega$ be the intersection of all $\Omega_x$ found for $x \in \mathbb{Q}^n$. Then $\mathbf{P}(\widetilde\Omega) = 1$, and we have
\begin{equation}\label{lem:lbddrational}
    \liminf_{\epsilon\rightarrow 0} 
    \inf_{0\leq t \leq T} \,[u_{\epsilon}(t,x,\omega) - \overline{u}(t,x)] \ge 0, \qquad \forall x\in \mathbb{Q}^n,\, \forall \omega \in \widetilde\Omega.
\end{equation}

To strengthen this pointwise lower bound to the locally uniform one in \eqref{eq:uflbdd}, we rely on the regularity results in Theorem \ref{thm:regularity}. Fix $\omega \in \widetilde\Omega$, and then fix $T,R>0$. In view of \eqref{eq:regx} for $u_\eps$ and the fact that the solution $\overline{u}$ of the homogenized HJ equation \eqref{hHJB} enjoys Lipschitz regularity (see Theorem 3.18 of \cite{MR4328923}), for any $\eta > 0$, there is an integer $k\in \mathbb{N}$ sufficiently large so that, for all $x\in \R^n$, we can find $x'\in 2^{-k} \mathbb{Z}^n$, such that
\begin{equation*}
    \sup_{t\in [0,T]} |u_\eps(t,x,\omega) - u_\eps(t,x',\omega)| + \sup_{t\in[0,T]} |\overline{u}(t,x)-\overline{u}(t,x')| < \eta.
\end{equation*}
Moreover, for $|x|\le R$, $x'$ can be chosen from the set $E_k := B_{R+1}(0)\cap 2^{-k}\mathbb{Z}^n$. Hence,
\begin{equation*}
    \inf_{t\in [0,T]} \inf_{|x|\le R} \left(u_\eps(t,x,\omega) - \overline{u}(t,x)\right) \ge - \eta + \inf_{x'\in E_k} \inf_{t\in[0,T]} \left(u_\eps(t,x',\omega)-\overline{u}(t,x')\right) .
\end{equation*}
Since $E_k$ is a finite set, passing $\eps \to 0$ in the inequality above, we get
\begin{equation*}
    \liminf_{\eps \to 0} \inf_{t\in [0,T]} \inf_{|x|\le R} \left(u_\eps(t,x,\omega) - \overline{u}(t,x)\right) \ge -\eta.
\end{equation*}
This immediately establishes \eqref{eq:uflbdd}.
\end{proof}

\section{Approximate Super Solutions}
\label{sec:convex}

The proof of the upper bound that leads to \eqref{eq:main} is more involved. The idea is to construct approximate super-solutions to the HJB equation. For this purpose and following \cite{KRV06}, we construct, for each $p\in \R^n$, an $\mathbb{R}^n$-valued random vector $v \in L^{\beta}(\Omega;\R^n)$ that can be viewed as a gradient, and integrating it yields approximate supersolutions. More precisely, we aim to find $v = v_p$ satisfying  
\begin{equation}
\label{eq:sublifted}
    \mathfrak{J}v(\omega) + H(p + v(\omega), \omega) \leq \overline{H}(p),
\end{equation}
where $\mathfrak{J}$ is the non-local divergence operator defined in \eqref{eq:JJ}. 

To understand the importance of this inequality, consider the random vector field
$$
\widetilde v(x,\omega) = v(\tau_x \omega).
$$
Using the relation $\mathscr{J}\widetilde v(x,\omega) = \mathfrak{J}v (\tau_x\omega)$, the above inequality can be transformed, formally, to
\begin{equation*}%\label{supersol}
    \mathscr{J}\widetilde v(x,\omega) + H(p + \widetilde v(x,\omega), \tau_x\omega) \leq \overline{H}(p).
\end{equation*}
If $\widetilde{v}$ is a regular vector field on $\R^n$ and is curl-free, we can define a potential field $V$ via the line integral
\begin{equation}
    \label{eq:vtoV}
    V(x,\omega) := \int_{0\rightsquigarrow x} \widetilde v(z(s),\omega)\cdot \mathrm{d}z(s),
\end{equation}
where $0\rightsquigarrow$ denotes a smooth simple path connecting $0$ to $z$. Then $V(0,\omega)=0$ and $\nabla V(x,\omega) = \widetilde{v}(x,\omega) = v(\tau_x \omega)$. From the inequality of $\widetilde{v}$ and the relation $\mathscr{L}V=\mathscr{J}\widetilde{v}$ we see, intuitively, $\widehat{u}_{\epsilon}(t,x,\omega):= \langle p, x \rangle + t\overline{H}(p) + \epsilon V(x/\epsilon ,\omega)$ would satisfy
\begin{equation*}
    \partial_{t}\widehat{u}_{\epsilon} \geq \epsilon^{\alpha-1} \mathscr{L}\widehat{u}_{\epsilon} + H(\nabla\widehat{u}_\eps, \tau_{x/\eps}\omega).
\end{equation*}
So in some sense $\widehat{u}_\eps$ is a super-solution of the HJB equation \eqref{HJB} with initial value $u_0(x)=\langle p,x\rangle$, and this super-solution would be used to get the upper bound for the conclusion of \eqref{eq:main} in that case of affine initial data. 

%%%%%%%%
\subsection{Convex analysis} We follow the method in \cite{KRV06} for the desired construction of $v$. The starting point is a convex analysis of the definition of $\ol{H}$ in \eqref{EffH}. 

We first realize the space $\mathscr{E}$ in \eqref{EffH} as the limit of an increasing sequence. Recall that $\Omega$ is separable and, for $r \in [1,\infty)$, $L^r(\Omega)$ is separable. Let $L^1_+(\Omega)$ denote the convex subset of $L^1(\Omega)$ consisting of non-negative functions. Then $L^1_+(\Omega)$ remains separable and there is a sequence $\{\varphi_j\}_{j\in \mathbb{N}^*}$ that is a dense subset of the space. We may assume that $\varphi_j\in \mathscr{U}$ for all $j$ (see Lemma \ref{Mollify}). For each $k,r \in \mathbb{N}^*$, let
\begin{equation*}
    \mathbf{D}_{k}:=\left\{ \textstyle\sum_{j=1}^{k}t_{j}\varphi_{j} \,:\, \sum_{j=1}^{k}t_{j} =1, t_{j}\in [0,1]\right\}
\end{equation*}
be the convex hull of $\{\varphi_{j}: 1\leq j \leq k\}$, and let
\begin{equation*}
\mathbf{B}_{r} = \{b \in \mathbf{B} \,:\, \sup_{\omega}|b(\omega)| \leq r\}
\end{equation*}
where $\mathbf{B}$ is defined in \eqref{eq:vfB}. 

Note that each $\mathbf{D}_{k}$ is strongly compact in $L^{1}(\Omega)$, and each $\mathbf{B}_{r}$ is weak compact in $L^\gamma(\Omega)$ for $1 \leq \gamma < \infty$. Define
\begin{align*}
    & \mathscr{E}_{r,k} =  \{(b,\rho) \in \mathbf{B}_{r} \times \mathbf{D}_{k} \,:\,  \mathfrak{L}^*\rho = \mathfrak{D}\cdot(b \rho )\}, \\
    & \mathscr{E}_{k} =  \cup_{r}\mathscr{E}_{r,k}, \quad \text{and} \quad \mathscr{E} =  \cup_{k} \mathscr{E}_{k}.
\end{align*}
Then the unions in the second line are limits of increasing sequences. 

\medskip

\begin{lem}
Fix $p\in\mathbb{R}^n$. For each $k \geq 1$, there exists a random variable $u_{k} \in \mathscr{U}$ such that
\begin{equation}\label{UpE1}
    \sup_{\rho\in\mathbf{D}_{k}}\int_{\Omega}(\mathfrak{L} u_{k} (\omega) + H_{k}(p + \mathfrak{D} u_{k}(\omega), \omega))\rho \,\mathrm{d}\mathbf{P} \leq \overline{H}(p) +\frac{1}{k},
\end{equation}
where 
\begin{equation*}
    H_{r}(p,\omega)=\sup_{|q|\leq r}[\langle p, q \rangle - L(q,\omega)].
\end{equation*}
\end{lem}
\begin{proof}
By the definition of $\overline{H}(p)$ and $\mathscr{E}_{r,k}$, we have
\begin{equation}\label{HpE}
    \overline{H}(p) \geq \sup_{(b,\rho)\in\mathscr{E}_{r,k}}\;\int_{\Omega} [\langle p,b(\omega)\rangle - L(b(\omega),\omega)]\rho(\omega) \mathrm{d}\mathbf{P}.
\end{equation}
Consider $\mathfrak{A}_{b}u = \mathfrak{L}u + \langle b, \mathfrak{D} u \rangle$ for $u\in \mathscr{U}$. Note that
\begin{equation*}
  \inf_{u\in \mathscr{U}}\int_{\Omega} (\mathfrak{A}_{b}u)\rho(\omega) \mathrm{d}\mathbf{P}= \left\{
   \begin{aligned}
   & 0,\quad & &\text{if}\ (b,\rho) \in \mathscr{E}_{r,k}, \\
   & -\infty, \quad & &\text{otherwise}.
   \end{aligned}
   \right.
\end{equation*}
We can hence insert the term $(\mathfrak{A}_bu)\rho$ to the integral in \eqref{HpE} and obtain
\begin{equation*}
    \overline{H}(p) \geq \sup_{\rho \in \mathbf{D}_{k}}\sup_{b\in \mathbf{B}_{r}}\inf_{u\in \mathscr{U}} \int_{\Omega} [\langle p,b(\omega)\rangle - L(b(\omega),\omega) + \mathfrak{A}_bu]\rho(\omega) \mathrm{d}\mathbf{P}.
\end{equation*}
Note that $\mathbf{B}_{r}$ is convex and weakly compact. Moreover, for each fixed $\rho \in \mathbf{D}_{k}$ the functional
\begin{equation*}
    (b,u) \rightarrow \int [\langle p,b(\omega)\rangle - L(b(\omega),\omega)+ \mathfrak{A}_bu(\omega)]\rho(\omega) d\mathbf{P}
\end{equation*}    
is concave and upper-semicontinuous in $b\in \mathbf{B}_{r}$ in the weak topology, and is linear and continuous in $u\in\mathscr{U}$. By minimax theorem (see e.g. \cite{S1958}), we can interchange the orders of taking extrema in $\mathscr{U}$ and $\mathbf{B}_{r}$, and, get
\begin{align*}
    \overline{H}(p) \geq & \sup_{\rho \in \mathbf{D}_{k}}\inf_{u\in \mathscr{U}}\sup_{b\in\mathbf{B}_{r}} \int_{\Omega} [\langle p,b(\omega)\rangle - L(b(\omega),\omega) + \mathfrak{A}_{b}u]\rho(\omega) \mathrm{d}\mathbf{P} \\
 \geq & \sup_{\rho \in \mathbf{D}_{k}}\inf_{u\in \mathscr{U}} \int_{\Omega} (\mathfrak{L}u + H_{r}(p + \mathfrak{D} u, \omega))\rho(\omega) \mathrm{d}\mathbf{P}
\end{align*}
where the last line follows from the definition of $H_r$. 
Since the functional
\begin{equation*}
    (\rho, u) \rightarrow \int_{\Omega} (\mathfrak{L}u + H_{r}(p + \mathfrak{D} u, \omega))\rho(\omega) \mathrm{d}\mathbf{P}
\end{equation*}
is convex and continuous in $u\in \mathscr{U}$, is linear and continuous in $\rho \in \mathbf{D}_{k}$, and $\mathbf{D}_{k}$ is compact in $L^1(\Omega)$, another usage of the minimax theorem yields
\begin{equation*}
     \overline{H}(p) \geq \inf_{u\in \mathscr{U}}\sup_{\rho \in \mathbf{D}_{k}} \left[ \int_\Omega (\mathfrak{L}u + H_{r}(p + \mathfrak{D} u, \omega))\rho(\omega) \mathrm{d}\mathbf{P}\right].
\end{equation*}
This yields the desired result.
\end{proof}

Fix $p\in \R^n$. Let $\{u_k\}$ be determined above and let $v_k = \mathfrak{D}u_k$. Our goal is to extract a subsequence of $v_k$ that has weak limit $v$ and show that $v$ satisfies \eqref{eq:sublifted}. This is the content of the next lemma.

\begin{lem}\label{Gardv}
The sequence $\{v_{k}\}$ is weakly compact in $L^{\beta}(\Omega;\R^n)$, and any weak limit $v = (v^1,\dots,v^n) \in L^{\beta}(\Omega;\R^n)$ is a zero mean and curl-free vector field in the following sense,
$$
\mathbf{E}[v] = 0, \quad \mathfrak{D}\times v = \left(\mathfrak{D}_j v^i - \mathfrak{D}_i v^j\right)_{i,j} = 0.
$$
and $v$ satisfies \eqref{eq:sublifted} almost surely. 
\end{lem}

\begin{proof}
By the definition of $v_k$, we have  $\mathfrak{L}u_k = \mathfrak{J}(v_k)$. From (\ref{UpE1}) it follows that
\begin{equation}\label{UpE12}
    \sup_{\rho\in\mathbf{D}_{k}}\int_{\Omega}\left(\mathfrak{J}v_{k} (\omega) + H_{k}(p + v_{k}(\omega), \omega)\right)\rho(\omega)\, \mathrm{d}\mathbf{P} \leq \overline{H}(p) +\frac{1}{k}.
\end{equation}
The exact argument in Theorem 5.2 of \cite{KRV06}, which uses only the uniform super-linearity assumption $p \rightarrow H(p,\omega) \geq c_{1}|p|^{\beta}-c_{2}$, shows that the sequence $\{v_k\}_{k\in \mathbb{N}^*}$ is uniformly integrable. Then the Dunford-Pettis theorem (see \cite[Theorem 4.7.18]{B2006}) yields that the sequence $\{v_{k}\}$ is weakly compact in $L^{\beta}(\Omega)$. The mean-zero and curl-free properties of $v_k$ are preserved for $v$ in the weak limit. %Thus any weak limit $v$ it is a curl-free vector field $\nabla\times v =0$, and it satisfies $\mathbb{E}^{\mathbb{P}}v = \lim_{k\rightarrow \infty}\mathbb{E}^{\mathbb{P}}v_{k} = 0$. 

We only show that the equation \eqref{eq:sublifted} holds for $v$. Note that the convexity of $H_{k}$ guarantees that for every fixed $\rho \in  \mathbf{D}_{k}$, the functional
\begin{equation}
    v \to \int_{\Omega} \left(\mathfrak{J} v (\omega) + H_{k}(p + v(\omega), \omega)\right)\rho(\omega) \mathrm{d}\mathbf{P}
\end{equation}
is weakly lower semicontinuous. Moreover, $\{ \mathbf{D}_{k} \}$ are increasing sequences of sets, and $H_{k} \nearrow H$ as $k \rightarrow \infty$. Thus for any weak limit $v$ of a convergent subsequence (still denote $\{v_{k}\}$), we can take the limit as $k\rightarrow \infty$ in (\ref{UpE12}) and obtain the bound
\begin{equation*}
    \sup_{\rho\in \cup_{k}\mathbf{D}_{k}}\int_{\Omega}(\mathfrak{J}v (\omega) + H(p + v(\omega), \omega))\rho(\omega) \mathrm{d}\mathbf{P} \leq \overline{H}(p).
\end{equation*}
Since $\cup_{k}\mathbf{D}_{k}$ is dense in $L^1_{+}(\Omega)$, the weak limit $v$ also satisfies the almost sure bound  
\begin{equation*}
    \mathfrak{J}v (\omega) + H(p + v(\omega), \omega) \leq \overline{H}(p).
\end{equation*}
This completes the proof.
\end{proof}

%%%%%
\begin{lem}\label{Ecv} Given $p\in \R^n$, let $v$ be the $L^\beta(\Omega)$ random vector determined in Lemma \ref{Gardv}, and let $\widetilde{v}$ be the random field $\widetilde{v}(x,\omega)=v(\tau_x \omega)$. Then $\widetilde{v}$ is in $L^\beta_{\rm loc}(\R^n)$ uniformly in $\Omega$ in the following sense: for any $R\ge 1$ there is a constant $C_{R} < \infty$, which depends only on $R,p,c_1,c_2,\beta$, such that 
\begin{equation}
\label{eq:unifbddv}
     \sup_{x\in\mathbb{R}^n}\int_{B_R(x)} |\widetilde v(y,\omega)|^{\beta} \,\mathrm{d}y  \le C_{R,p}
\end{equation}
holds $\mathbf{P}$-almost surely.
\end{lem}
\begin{proof} Let $\phi(x) = c(\sqrt{1+|x|^2})^{-n-\delta}$ for some $\delta>0$ to be fixed later and $c>0$ is chosen so that $\int_{\R^n} \phi =1$. In view of \eqref{eq:sublifted} and assumption $(\bf{A})$, we get
\begin{equation}
    \int_{\mathbb{R}^n}(\mathscr{J}\widetilde v (x,\omega) + c_{1}|p + \widetilde v(x,\omega)|^{\beta} - c_{2})\phi(x)\,\mathrm{d}x \leq \overline{H}(p).
\end{equation}
Recall that $\mathscr{J}\widetilde v=\nabla \cdot (\mathcal{R}\widetilde v)$ and that $\mathcal{R}$ is a regular integral operator; see \eqref{eq:cRdef}. Using integration by parts and some basic inequalities, we get
\begin{equation}\label{vLpE}
    \int_{\mathbb{R}^n} |\widetilde v(x,\omega)|^{\beta}\phi(x)\,\mathrm{d}x  
    \leq C_{1} \int_{\mathbb{R}^n}  \mathcal{R} \widetilde v (x,\omega)\cdot \nabla \phi(x)\,\mathrm{d}x + C_{2},
\end{equation}
where $C_1 = C_1(c_1,\beta)$ and $C_2(c_1,\beta,p)$ are constants that do not depend on $\omega$. We claim that $|\mathcal{R}^*\nabla \phi(x)|$ is bounded on $\R^n$ and there exists a constant $C_3$ depending only on $\phi$ (and hence on $c$ on $n$) such that for all $x$, 
\begin{equation}
\label{eq:bddR*phi}
    \left|\mathcal{R}^* \nabla \phi(x)\right| \le \frac{C_3}{|x|^{n+(\alpha-1)}}.%\frac{C_3}{|x|^{n+\delta+(\alpha-1)}}.
\end{equation}

With the uniform boundedness and the estimate above for $\mathcal{R}^*\nabla \phi$, and using the expression of $\phi$, we easily check
\begin{equation*}
\begin{aligned}
    C_4 := &\int_{\R^n} \left|\mathcal{R}^*\nabla \phi(x)\right|^{\beta'} \frac{1}{|\phi(x)|^{\beta'/\beta}}\,\dd x\\
    \le & \|\mathcal{R}^*\nabla \phi\|_{L^\infty}|B_1(0)| + C\int_{\{|x|\ge 1\}} \left(\frac{1}{|x|^{n+\alpha-1}} (\sqrt{1+|x|^2})^{\frac{n+\delta}\beta}\right)^{\beta'}\,\dd x\\
    < &\infty,
    \end{aligned}
\end{equation*}
as long as $\alpha > 1+\frac{\delta}{\beta}$ (i.e., $\delta < \beta(\alpha-1)$). Given $\alpha \in (1,2)$ and $\beta >1$, we can always choose $\delta>0$ so that this condition holds. Then applying H\"older's and then Young's inequalities, we get
\begin{equation*}
\begin{aligned}
    J := &\left|\int_{\mathbb{R}^n}  \mathcal{R} \widetilde v (x,\omega)\cdot \nabla \phi(x)\,\mathrm{d}x\right| = \left|\int_{\mathbb{R}^n} \widetilde v (x,\omega)\cdot \mathcal{R}^*\nabla \phi(x)\,\mathrm{d}x\right|\\
    \le & \left(\int_{\R^n} |\widetilde v(x)|^\beta \phi(x)\,\dd x\right)^{\frac1{\beta}}\left(\int_{\R^n} |\mathcal{R}^*\nabla \phi(x)|^{\beta'} \frac{1}{|\phi(x)|^{\beta'/\beta}}\,\dd x\right)^{\frac{1}{\beta'}}\\
    \le &\frac{1}{2C_1} \int_{\R^n}|\widetilde v(x,\omega)|^\beta \phi(x)\,\upd x + C_4\frac{1}{\beta'}(2C_1/\beta)^{\beta'/\beta}.
    \end{aligned}
\end{equation*}
Combining this estimate with \eqref{vLpE}, we get a constant $C_5=C_5(p,n,c_1,c_2,\alpha,\beta)$ such that
\begin{equation}
    \int_{\R^n} |\widetilde{v}(y,\omega)|^\beta \phi(y)\,\dd y \le C_5.
\end{equation}
This yields, for each $R>0$, a constant $C_R$ depending on $C_5$ and $R$ such that
\begin{equation*}
    \int_{B_R(0)} |\widetilde{v}(y,\omega)|^\beta \,\dd y \le C_R.
\end{equation*}
Finally, translating $\phi$ to $\phi(\cdot -x)$, we get the desired estimate \eqref{eq:unifbddv}.

\medskip

It remains to check the estimate \eqref{eq:bddR*phi}, and for this it suffices to consider $\mathcal{R}^*\partial_1 \phi$ as the other components are bounded in the same way. Let $\psi = \partial_1\phi$ and verify that $\psi$ is uniformly bounded and $\psi(x) \sim |x|^{-n-\delta-1}$ for large $|x|$. By the explicit form of the kernel of $\mathcal{R}$ in \eqref{eq:cRdef}, we see that for all $x\in \R^n$
\begin{equation*}
\begin{aligned}
    |\mathcal{R}^*\psi(x)| &= \left|\int_{\R^n} R(y-x)\psi(y)\,\dd y\right|\\ 
    &\le C\int_{\R^n} \frac{1}{|x-y|^{n+\alpha-2}}|\psi(y)|\,\dd y \le C(\|\psi\|_{L^\infty(\R^n)} + \|\psi\|_{L^1(\R^n)}).
    \end{aligned}
\end{equation*}
This verifies the uniform boundedness of $|\mathcal{R}^*\psi|$. Note that here and in the rest of the proof, the bounding constant $C$ will vary from line to line.

To prove the rate of $|\mathcal{R}^*\psi|$  for $|x|$ large, say $|x|\ge 4$, we first write the definition of $\mathcal{R}^*\psi$ as the sum of $I_1+I_2$ as follows:
\begin{equation*}
    I_1 = \int_{B_{|x|}(x/2)} R(y-x)\psi(y)\,\dd y,
\end{equation*}
and
\begin{equation*}
    I_2 = \int_{\R^n\setminus B_{|x|}(x/2)}  R(y-x)\psi(y)\,\dd y.
\end{equation*}

For $I_2$, we note that outside $B_{|x|}(x/2)$, 
\begin{equation*}
    \frac{|y|}{|y-x/2|}, \; \frac{|y-x|}{|y-x/2|} \quad \in \quad [\frac12, \frac32].
\end{equation*}
We then get
\begin{equation*}
    I_2 \le C\int_{\R^n\setminus B_{|x|}(x/2)} \frac{1}{|y-x/2|^{n+\alpha-2}} \frac{1}{|y-x/2|^{n+\delta+1}}\,\dd y \le C\frac{1}{|x|^{n+\delta+\alpha-1}}.
\end{equation*}

For $I_1$, we further break the integration domain into two regions. On $B_{|x|}(x/2)\setminus B_{|x|/2}(0)$, we have $|y|\ge |x|/2$, and hence $|\psi(y)|\le C|x|^{-n-\delta-1}$. We get
\begin{equation*}
\begin{aligned}
    \left| \int_{B_{|x|}(x/2)\setminus B_{|x|/2}(0)} R(y-x)\phi(y) \dd y\right| &\le C\frac{1}{|x|^{n+\delta+1}} \int_{B_{2|x|}(x)} |R(y-x)|\,\dd y\\
    &\le C\frac{1}{|x|^{n+\delta+1}} \int_{B_{2|x|}(0)} \frac{1}{|y|^{n-2+\alpha}}\,\dd y\\
    &\le C\frac{1}{|x|^{n+\delta+\alpha-1}}.
    \end{aligned}
\end{equation*}
For the remaining part of $I_1$, note that $R(y-x)$ and $\psi(y)$ are both smooth functions in this region (there is no singularity), and an integration by parts shows
\begin{equation*}
    \int_{B_{|x|/2}(0)} R(y-x)\psi(y)\,\dd y = \int_{\partial B_{|x|/2}(0)} n_1(y)R(y-x)\phi(y)\,\dd y - \int_{B_{|x|/2}(0)} \partial_1 R(y-x)\phi(y)\,\dd y.
\end{equation*}
For the integral on the sphere where by an abuse of notation $\dd y$ means the surface measure, we note $|n_1|\le 1$ and $|y-x|\ge |x|/2$. Hence on the sphere, $|R(y-x)|\le C|x|^{-n-\alpha+2}$ and $|\phi(y)|\le C|x|^{-n-\delta}$. It follows that
\begin{equation*}
    \left|\int_{\partial B_{|x|/2}(0)} n_1(y)R(y-x)\phi(y)\,\dd y\right| \le C\frac{1}{|x|^{n+\delta+\alpha-1}}.
\end{equation*}
For the other integral over $B_{|x|/2}(0)$, we note that $|y-x|>|x|/2$, and hence by the homogeneity of $\nabla R$ (of degree $-n-\alpha+1$), 
\begin{equation*}
    \left|\int_{B_{|x|/2}(0)} \partial_1 R(y-x)\phi(y)\,\dd y\right| \le C\frac{1}{|x|^{n+\alpha-1}} \int_{B_{|x|/2}(0)} \phi(y)\,\dd y \le C\frac{1}{|x|^{n+\alpha-1}}.
\end{equation*}
Combining all those estimates above, we verify \eqref{eq:bddR*phi} and complete the proof of the lemma.
\end{proof}

\begin{rem} By an application of Fubini's theorem, one easily sees that, for any random vector $f\in L^\beta(\Omega)$, the random field $\widetilde f(x,\omega) = f(\tau_x \omega)$ belongs to $L^\beta_{\rm loc}(\R^n)$, $\mathbf{P}$-a.s. The lemma above, however, yields a stronger result under the given assumptions: it provides local $L^\beta$ estimates for the random field $\widetilde v(x,\omega) = v(\tau_x \omega)$ that are uniform in $\omega$. 
\end{rem}

\medskip 

To construct a vector field on $\R^n$ out of $v$, we first mollify it by \eqref{eq:rvmolify}. For $\delta \in (0,1)$, let
\begin{equation}\label{Mvcov}
    v_{\delta}(\omega) := \frac{1}{\delta^n} \int_{\mathbb{R}^n} \rho(y/\delta)v(\tau_y\omega)\,\dd y, 
\end{equation}
where $\rho\in C^\infty_c(\R^n;[0,\infty))$ is a smooth even function supported in the unit ball $B_1(0)$ and satisfies $\int_{\mathbb{R}^n} \rho (x)\,\mathrm{d}x =1$. Let $\widetilde v(\cdot,\omega)$ and $\widetilde v_\delta(\cdot,\omega)$ be the vector field on $\R^n$ obtained by the translation operator $\{\tau_x\}$. We check directly that 
\begin{equation*}
    \widetilde{v}_\delta(x,\omega) = \frac{1}{\delta^n} \int_{\R^n} \rho\left(\frac{x-y}{\delta}\right) v(\tau_y\omega)\,\dd y = \rho_\delta * \widetilde{v}(x,\omega)
\end{equation*}
So, $\widetilde v_\delta(\cdot,\omega)$ is a smooth family of random vector fields approximating $\widetilde{v}(\cdot,\omega)$ as $\delta \to 0$. Moreover, the condition $\mathfrak{D}\times v = 0$ implies that, for each $\delta\in (0,1)$, the smooth vector field $\widetilde{v}_\delta(\cdot,\omega)$ is curl-free. As a result, the line integral
\begin{equation}\label{Mpot}
    V_{\delta}(x,\omega) = \int_{0 \rightsquigarrow x}\langle v_{\delta}(\tau_{z(s)}\omega), \mathrm{d}z(s)\rangle,
\end{equation}
rigorously defines a smooth potential field satisfying $V_\delta(0,\omega) = 0$, $\nabla V_\delta(x,\omega) = \widetilde{v}_\delta(x,\omega)$. Note that in the definition, the simple smooth path $0\rightsquigarrow$ can be arbitrary because $\widetilde{v}_\delta$ is curl-free.

Moreover, in view of the local $L^\beta$ estimate \eqref{eq:unifbddv} of $\widetilde{v}$ that is uniform in $\omega$, we get, for each $\delta \in (0,1)$, there is a constant $C_\delta<\infty$ so that
\begin{equation}
    \label{eq:vhatbdd}
    \|v_\delta(\omega)\|_{L^\infty(\Omega)} \le C_\delta.
\end{equation}
By the definition of $V_\delta(\cdot,\omega)$, we have $|V_\delta(x,\omega)| \le C_\delta|x|$ almost surely in $\Omega$. In view of the mean-zero and stationary property of $\widetilde v_\delta(\cdot,\omega)$ and by the ergodic theorem (see e.g., Lemma A.5 of \cite{MR2914944}), 
% also
% \begin{equation}\label{VdelE}
%     \sup_{\omega\in\Omega} |V_{\delta}(0,\omega)| \leq C_{\delta}.
% \end{equation}
%Moreover, the ergodicity of $\tau_{x}$ yields that 
for any $\delta >0$ and $l < \infty$, we have
\begin{equation}\label{ergoV}
    \lim_{\epsilon \to 0} \epsilon \sup_{|x|\leq l}|V_{\delta}(\epsilon^{-1}x,\omega)| = %\lim_{\epsilon \to 0} \epsilon \sup_{|x|\leq l}|V_{\delta}(0,\tau_{\epsilon^{-1}x}\omega)| = \sup_{|x|\leq l}\mathbb{E}v_{\delta}(0,\omega) = 
    0, \quad \mathbf{P}\text{-almost surely}.
\end{equation}
For $v_{\delta}$ and $V_{\delta}$, we have the following technical lemma.

\begin{lem}\label{Appv}
Assume that the assumptions {\upshape({\bf A1})({\bf A2})} and {\upshape({\bf B})} hold. For any $p\in \R^n$ and $\delta \in (0,1)$, there exists a random vector $\hat{v}_{\delta}$, satisfying $\mathbf{E}\hat{v}_{\delta} = 0$, $\mathfrak{D}\times \hat{v}_{\delta} = 0$, such that
\begin{equation}\label{IeHd}
    \mathfrak{J}\hat{v}_{\delta}(\omega) + H(p + \hat{v}_{\delta}(\omega), \omega) \leq \overline{H}(p) + \widetilde\gamma(\delta),
\end{equation}
where $\widetilde\gamma$ is a modulus of continuity satisfying $\widetilde\gamma(\delta)\to 0$ as $\delta\to 0$.
\end{lem}
\begin{proof} Let $\gamma$ and $M$ be the modulus of continuity and the bounding constant in \eqref{AHd} in assumption ({\bf B}). Given $p \in \R^n$ and $\delta\in (0,1)$, let $p_\delta = (1+\gamma(\delta))p$. For $p_\delta$ we apply Lemma \ref{Gardv} and find a random vector $v$ associated to $p_\delta$ which satisfies the inequality \eqref{eq:sublifted} for $p$ replaced by $p_\delta$. Then we integrate \eqref{eq:sublifted} against the mollifier $\rho_\delta$ as in \eqref{Mvcov}, and, find
\begin{equation*}
     \int_{\R^n} \mathfrak{J} v(\tau_y \omega) \rho_\delta(y)\,\dd y + \int_{B_1(0)}H(p_\delta + v(\tau_{\delta y}\omega), \tau_{\delta y}\omega)\rho(y)\,\mathrm{d}y \leq \overline{H}(p_\delta).
\end{equation*}
For the first integral, since $\mathfrak{J}$ commutes with $\{\tau_x\}$ and is a linear operator, this term becomes $\mathfrak{J}v_\delta$. For the second integral, we apply the inequality \eqref{AHd}. This leads to
\begin{equation*}
    \mathfrak{J} v_{\delta}(\omega) + (1+\gamma(\delta))\int_{B_1(0)}H((1+\gamma(\delta))^{-1}(p_\delta + v(\tau_{\delta y}\omega)), \omega)\rho(y)\,\mathrm{d}y \leq \overline{H}(p_\delta) + M\gamma(\delta).
\end{equation*}
Let $\hat{v}_{\delta}= v_{\delta}/(1+ \gamma(\delta))$, then it still satisfies the mean-zero and curl-free conditions. In view of the linearity of $\mathfrak{J}$ and the convexity of $H$, we check that $\hat{v}_{\delta}$ also satisfies
\begin{equation*}
    \mathfrak{J}\hat{v}_{\delta}(\omega) + H(p + \hat{v}_{\delta}, \omega) \leq \frac{\overline{H}(p_\delta) + M\gamma(\delta)}{1+\gamma(\delta)},
\end{equation*}
Finally, by the continuity of $\overline{H}(p)$ in $p$, we find easily a modulus of continuity $\widetilde\gamma: (0,1)\to [0,\infty)$ so that \eqref{IeHd} holds.
\end{proof}

\begin{rem} From the proof we see that the obtained $\hat{v}_\delta$ is regular in the sense that the vector field $x\mapsto \hat{v}_\delta(\tau_x \omega)$ is smooth. By abuse of notation, let $\hat{v}(x,\omega)$ denote this vector field, then $\hat{v}(\cdot,\omega)$ is curl-free, and \eqref{IeHd} is saying
\begin{equation}
\label{eq:vdeleq-1}
    \mathscr{J}\hat{v}_\delta(y,\omega) + H(p+\hat{v}_\delta(y,\omega),y,\omega) \le \ol{H}(p) + \widetilde\gamma(\delta).
\end{equation}
\end{rem}

Without assumption ({\bf B}), we still have the following alternative of Lemma \ref{Appv}.

\begin{lem}\label{vVcompaison}
Assume that the assumptions $(\bf{A1})$ and $(\bf{A2})$ hold. Suppose that $p\in\mathbb{R}^{n}$, $\lambda \in \mathbb{R}$, $v \in L^{\beta}(\Omega;\R^n)$ satisfy $\mathbf{E}[v] = 0$, $\mathfrak{D}\times v = 0$, and
\begin{equation}\label{supersol11}
    \mathfrak{J} v(\omega) + H(p + v(\omega), \omega) \leq \lambda.
\end{equation}
Let $v_\delta$ and $V_{\delta}$ be given by  \eqref{Mvcov} and \eqref{Mpot} respectively. Then, given a control $c(t,x)\in \mathscr{C}$, there exists a modulus of continuity $r_{c}(\delta)$ depending on $c$, i.e., $r_c(\delta) \to 0$ as $\delta \to 0$, such that the random processes
\begin{equation*}
    X_{t} := x + \int^{t}_{0} c(s,X_{s}) \mathrm{d}s + L^{\alpha}_{t} \quad\text{and}\quad \xi(t) := \int^{t}_0 L(c(s,X_{s}),\tau_{X_{s}}\omega)\mathrm{d}s
\end{equation*}
satisfy
\begin{align}\label{El44}
    \langle p, X_{t} - x \rangle - t\lambda - \xi(t) \leq & V_{\delta}(x,\omega) - V_{\delta}(X_{t},\omega) + \langle p, L^{\alpha}_{t} \rangle + r_{c}(\delta)t \nonumber\\
    & +\int_{0}^{t} \int_{\mathbb{R}^n} [V_{\delta}(X_{s^{-}}+ y,\omega)-V_{\delta}(X_{s^{-}},\omega)]\widetilde{N}(\mathrm{d}s,\,\mathrm{d}y),
\end{align}
and, for all $s>0$ and $z\in \R^n$,
\begin{equation}
    \label{eq:l54-main}
\mathfrak{J}v_\delta(\tau_z \omega) + \langle c(s,z),p+v_\delta(\tau_z \omega) \rangle \le \ol{H}(p) + L(c(s,z),\tau_z\omega) + r_c(\delta).
\end{equation}
\end{lem}

\begin{proof}
We note that, almost surely, $\widetilde{v}_\delta(x,\omega) = v_\delta(\tau_x \omega)$ is a smooth vector field, $V_\delta$ is a smooth function and $\nabla V_{\delta}(x,\omega) = \widetilde{v}_\delta(x,\omega)$. Applying the It\^{o}'s formula to the smooth function $V_\delta(x,\omega)+\langle p,x\rangle$, we get
\begin{align}\label{ItoVp}
     & V_{\delta}(X_{t},\omega) - V_{\delta}(x,\omega) + \langle p, X_{t} - x \rangle \nonumber \\
    =& \int_{0}^{ t} \langle c(s, X_{s}), p+ \widetilde{v}_{\delta}(X_{s}, \omega)\rangle \mathrm{d}s  + \langle p, L^{\alpha}_{t} \rangle + \int_{0}^{t} \int_{\mathbb{R}^n} [V_{\delta}(X_{s^{-}}+ y)-V_{\delta}(X_{s^{-}})]\widetilde{N}(\mathrm{d}s,\,\mathrm{d}y) \nonumber \\
     & +\int_{0}^{t} \mathscr{L} V_{\delta}(X_{s^{-}},\omega)\mathrm{d}s \nonumber \\
     =& \int_{0}^{ t} \langle c(s, X_{s}), p+ \widetilde{v}_{\delta}(X_{s}, \omega)\rangle \mathrm{d}s + \langle p, L^{\alpha}_{t} \rangle + \int_{0}^{t} \int_{\mathbb{R}^n} [V_{\delta}(X_{s^{-}}+ y)-V_{\delta}(X_{s^{-}})]\widetilde{N}(\mathrm{d}s,\,\mathrm{d}y) \nonumber \\
     & + \int^{t}_{0} \mathscr{J} \widetilde{v}_{\delta}(X_{s},\omega)\mathrm{d}s.
\end{align}
In the third line above we used the fact $\mathscr{L} V_\delta = \mathscr{J}( \widetilde{v}_\delta)$. Using the duality relation $ \langle p, q\rangle \leq H(p,\omega) + L(q, \omega)$, $p,q \in\mathbb{R}^n$, we have
\begin{equation}
\label{eq:l54-1}
\begin{aligned}
    \langle c(s, X_{s}), p+ \widetilde{v}_{\delta}(X_{s},\omega)\rangle \
    &=  \int_{\mathbb{R}^n} \langle c(s, X_{s}), p+ \widetilde{v}(y,\omega)\rangle\rho_{\delta}(X_{s}-y)\,\mathrm{d}y \\
    &\leq  \int_{\mathbb{R}^n} (L(c(s, X_{s}), \tau_{y}\omega)+H(p+ \widetilde{v}(y,\omega), \tau_{y}\omega))\rho_{\delta}(X_{s}-y)\,\mathrm{d}y.
\end{aligned}
\end{equation}
Plugging the above into \eqref{ItoVp}, and using \eqref{supersol11}, we have
\begin{equation*}
\begin{aligned}
     & V_{\delta}(X_{t},\omega) - V_{\delta}(x,\omega) + \langle p, X_{t} - x \rangle \nonumber \\
     \leq  & \langle p, L^{\alpha}_{t} \rangle + \int_{0}^{t}\int_{\mathbb{R}^n} L(c(s, X_{s}), \tau_{y}\omega)\rho_{\delta}(X_{s}-y)\,\mathrm{d}y\dd s \nonumber \\
     & + \int_{0}^{t} \int_{\mathbb{R}^n} [V_{\delta}(X_{s^{-}}+ y)-V_{\delta}(X_{s^{-}})]\widetilde{N}(\mathrm{d}s,\mathrm{d}y) \nonumber \\
     & + \int_{0}^{t}\int_{\mathbb{R}^n} [\mathscr{J} \widetilde{v}(y,\omega) + H(p + \widetilde{v}(y,\omega), \tau_y\omega)]\rho_{\delta}(X_{s}-y)\,\mathrm{d}y \nonumber\\
     \leq & t\lambda  + \langle p, L^{\alpha}_{t} \rangle + \int_{0}^{t}\int_{\mathbb{R}^n} L(c(s, X_{s}), \tau_{y}\omega)\rho_{\delta}(X_s-y)\,\mathrm{d}y\dd s \nonumber\\
     & + \int_{0}^{t} \int_{\mathbb{R}^n} [V_{\delta}(X_{s^{-}}+ y)-V_{\delta}(X_{s^{-}})]\widetilde{N}(\mathrm{d}s,\mathrm{d}y).
\end{aligned}
\end{equation*}
Hence we get
\begin{align}
    &  \langle p, X_{t} - x \rangle - t\lambda - \xi(t) \nonumber \\
    \leq & V_{\delta}(x,\omega)- V_{\delta}(X_{t},\omega) + \langle p, L^{\alpha}_{t} \rangle + \int_{0}^{t} \int_{\mathbb{R}^N} [V_{\delta}(X_{s^{-}}+ y)-V_{\delta}(X_{s^{-}})]\widetilde{N}(\mathrm{d}s,\,\mathrm{d}y) \nonumber \\
    & + \int^{t}_{0}\int_{\mathbb{R}^n} (L(c(s, X_{s}), \tau_{y}\omega)- L(c(s, X_{s}), \tau_{X_{s}}\omega))\rho_{\delta}(X_{s}-y)\,\mathrm{d}yds. 
\end{align}
For the last term, for each $s\in (0,t)$, the integral over $\R^n$ is reduced to the integral in $B_\delta(X_{s})$. Given $c\in \mathscr{C}$, $\|c\|_{L^\infty}$ is finite. Using the uniform continuity of $L(q,\tau_{x}\omega)$ in $x$ (and locally in $q$), i.e., by Assumption ({\bf{A2}}), we can bound the last term by
\begin{align}\label{Remainder}
        & \int^{t}_{0}\int_{\mathbb{R}^n} (L(c(s, X_{s}), \tau_{y}\omega)- L(c(s, X_{s}), \tau_{X_{s}}\omega))\rho_{\delta}(X_{s}-y)\,\mathrm{d}yds \nonumber \\
        \leq & t\sup_{|q| \leq \|c\|_{L^{\infty}}}\sup_{\omega}\sup_{|y| \leq \delta}|L(q,\tau_{y}\omega) - L(q,\omega)| := tr_{c}(\delta).
\end{align}
This proves \eqref{El44}.

To prove \eqref{eq:l54-main}, we apply the inequality \eqref{supersol11} at each $\tau_y\omega$, $y\in \R^n$, and multiply the inequality by $\rho_\delta(z-\cdot)$ and integrate. This leads to
\begin{equation*}
    \mathfrak{J}v_\delta(\tau_z\omega) \le \ol{H}(p) - \int_{\R^n} H(p+v_\delta(\tau_y\omega),\tau_y\omega)\rho_\delta(z-y)\,\dd y.
\end{equation*}
Combine this with \eqref{eq:l54-1} (which still holds for $X_{s^-}$ replaced by $z$) to get
\begin{equation*}
\begin{aligned}
    \mathfrak{J}v_\delta(\tau_z\omega) &+ \langle c(s,z),p+v_\delta(\tau_z\omega)\rangle \le \ol{H}(p) + L(c(s,z),\tau_z\omega) \\
    &+ \int_{\R^n} (L(c(s,z),\tau_y\omega)-L(c(s,z),\tau_z\omega))\rho_\delta(z-y)\,\dd y.
    \end{aligned}
\end{equation*}
Using the modulus of continuity $r_c$ defined in \eqref{Remainder}, we see the proof is hence completed.
\end{proof}

%%%%%%%%%%%%%%
\section{Proof of the Upper Bound}
\label{sec:proofUB}

In this section, we complete the proof of Theorem \ref{thm:main}. The assumptions ({\bf A1})({\bf A2}) and ({\bf A3}$'$) are invoked. The lower bound is already proved in Section \ref{sec:lowerbdd}. Under each of the two different extra assumptions, Assumptions $({\bf B})$ and $({\bf C})$, we establish the upper bound which then completes the proof of the main theorem.

%Note that, in view of the regularity in $(t,x)$ of $u_\eps$ and by the same argument in the proof of Theorem \ref{LBR} for the lower bound, to complete the proof of Theorem \ref{thm:main} it suffices to establish the upper bound for each fixed $x$, that is, to prove $\mathbf{P}$-almost surely, 
%\begin{equation*}
%    \limsup_{\epsilon \to 0} \sup_{t\in [0,T]} u_\epsilon(t,x)-\overline{u}(t,x) \le 0.
%\end{equation*}
We have the following representation formulas for $u_\eps(t,x)$ and $\overline{u}(t,x)$:
\begin{equation*}
    \begin{aligned}
        &u_\epsilon(t,x) = \sup_{c\in \mathscr{C}^*} \mathbb{E}^{Q^{c,\epsilon}_x} \left[u_0(X^\epsilon_t) - \xi^c_\epsilon(t)\right], \qquad \xi^c_\epsilon(t) = \int_0^t L(c(\epsilon^{-1}s, \epsilon^{-1}X^\epsilon_s), \tau_{\epsilon^{-1}X^\epsilon_s} \omega)\,\mathrm{d}s.\\
        & \overline{u}(t,x) = \sup_{y\in \R^n} \left[u_0(y) - t\overline{L}\left(\frac{y-x}{t}\right)\right].
    \end{aligned}
\end{equation*}
Here $X^\epsilon_t$ is the jump-diffusion process \eqref{eq:Xepsc} determined by $c$ starting from $x$. Note that by Lemma \ref{Cset} the control set $\mathscr{C}$ in the original control formula (\ref{vf}) is already replaced by the subset $\mathscr{C}^{\ast}$ defined in \eqref{sCset}. As remarked before, we may assume $u_0$ is uniformly bounded and Lipschitz. 

The difference between $u_\eps$ and $\ol u$ satisfies
\begin{equation}
\label{eq:diffform}
      u_{\epsilon}(t,x,\omega) - \ol u(t,x)
    =  \sup_{c\in\mathscr{C}^{\ast}} \mathbb{E}^{Q^{c,\epsilon}_{x} }\left( u_{0}(X^{\epsilon}_{t}) - \xi^c_{\epsilon}(t) \right) - \sup_{y}(u_{0}(y)- t \overline{L}((y-x)/t)). 
%    \leq & \sup_{c\in\mathscr{C}^{\ast}} \mathbb{E}^{Q^{c,\epsilon}_{x} } \left( t \overline{L}((X^\eps_t-x)/t)) - \xi^c_{\epsilon}(t) \right)\\
 %   = & \sup_{c\in\mathscr{C}^{\ast}} \mathbb{E}^{Q^{c,\epsilon}_{x} }\Big(\sup_{p\in\mathbb{R}^n}[\langle p, X^{\epsilon}_{t} -x \rangle - t\overline{H}(p)]- \xi^c_{\epsilon}(t)\Big).
\end{equation}

\medskip 

In view of the condition \eqref{A2L} from assumption ({\bf{A}}), for every control $c(t,x)\in \mathscr{C}^{\ast}$, the solution $X^\eps_t$ of the SDE \eqref{eq:Xepsc} and the cost $\xi^c_\epsilon$ along it defined by \eqref{eq:xiepsdef} satisfy
\begin{equation}
\begin{aligned}
&\inf_{x,\omega} \xi^c_{\epsilon} = \inf_{x,\omega} \int^{t}_{0} L(c(\epsilon^{-1}s,\epsilon^{-1}X^{\epsilon}_{s}),\tau_{\epsilon^{-1}X^{\epsilon}_{s}\omega})\mathrm{d}s \geq -c_{4}t, \\
&\sup_{x,\omega}\mathbb{E}^{Q^{c,\epsilon}_{x}}|\xi^c_{\epsilon}(t)|\leq \sup_{x,\omega}\mathbb{E}^{Q^{c,\epsilon}_{x}} c'_3\int^{t}_{0} |c(s,\epsilon^{-1}X^{\epsilon}_{s})|^{\beta'} \mathrm{d}s + c_2 t \leq  C(t + \epsilon^{1-\frac1\alpha} t^{\frac{1}{\alpha}}).
\end{aligned}
\end{equation}
Combining this with moment estimates for $L^{\alpha}_{t}$, we deduce that for every $1 < \gamma < (\beta' \wedge \alpha)$, and $\ell, T>0$,
\begin{align*}
    & \sup_{|x| \leq l}\sup_{0\leq t \leq T}\sup_{c\in \mathscr{C}^{\ast}}\sup_{\epsilon\in (0,1]}\mathbb{E}^{Q^{c,\epsilon}_{x}} |X^{\epsilon}_{t} - x|^{\gamma} \nonumber \\
    \leq & \sup_{|x| \leq l}\sup_{0\leq t \leq T}\sup_{c\in \mathscr{C}^{\ast}}\sup_{\epsilon\in (0,1]}\mathbb{E}^{Q^{c,\epsilon}_{x}}\left|\int^{t}_{0} c(s, \epsilon^{-1}X^{\epsilon}_{t}) \mathrm{d}s\right|^{\gamma} + \mathbb{E}\,|\epsilon^{1-\frac{1}{\alpha}}L^{\alpha}_t|^{\gamma} \nonumber \\
    < & \infty.
\end{align*}
Thus $|X^{\epsilon}_{t}|$ is uniformly integrable with respect to the probability measures in
\begin{equation}
\label{eq:UImeasset}
    \{Q^{c,\epsilon}_{x}: |x| \leq l, t\in [0,T], \epsilon\in (0,1], c\in \mathscr{C}^{\ast}\}.
\end{equation}
In view of the formula \eqref{eq:diffform} for $u_\eps-\ol{u}$, the Lipschitz boundedness of $u_0$, the locally uniform bound of $\ol{u}$, the uniform lower bound on $\xi^c_\eps$ above and the uniform integrability of $|X^\eps_t|$, we conclude that, for every $\delta >0$, there exists a constant $M(\delta)>0$ such that for any fixed $t \in [0,T]$, $x\in B_l(0)$ and for almost all $\omega$, 
\begin{equation*}
\begin{aligned}
     u_{\epsilon}(t,x,\omega) - \ol{u}(t,x) &\leq  \sup_{c\in\mathscr{C}^{\ast}} \mathbb{E}^{Q^{c,\epsilon}_{x} }\left[\mathbf{1}_{|X^\eps_t-x|\le M_\delta} \left(u_0(X^\eps_t)-\xi^c_\eps(t) - \ol{u}(t,x)\right)\right] + \delta\\
     &\leq  \sup_{c\in\mathscr{C}^{\ast}} \mathbb{E}^{Q^{c,\epsilon}_{x} }\left[\mathbf{1}_{|X^\eps_t-x|\le M_\delta} \left(-\xi^c_\eps(t) + t\ol{L}\left(\frac{X^\eps_t-x}{t}\right)\right)\right] + \delta\\
    &\leq  \sup_{c\in\mathscr{C}^{\ast}} \mathbb{E}^{Q^{c,\epsilon}_{x} }[(\sup_{p\in\mathbb{R}^n}[\langle p, X^{\epsilon}_{t} -x \rangle - t\overline{H}(p)]- \xi^c_{\epsilon}(t))\mathbf{1}_{ |X^{\epsilon}_{t}-x| \leq M(\delta)}] + \delta.
    \end{aligned}
\end{equation*}
Using the super-linear growth of $\ol{H}$ in $|p|$ again, the supremum over $p\in \R^n$ above is achieved in a compact set. Therefore, in order to obtain the upper bound, it is enough to show that for each fixed $p\in\mathbb{R}^n$, and arbitrary $l, T, \eta>0$,
\begin{equation}\label{uppE}
    \lim_{\epsilon\to 0}\sup_{c\in\mathscr{C}^{\ast}} \sup_{|x| \leq l} \sup_{0 \leq t \leq T} Q^{c, \epsilon}_{x} \{ \langle p, (X^{\epsilon}_{t} -x) \rangle - t\overline{H}(p)- \xi^c_{\epsilon}(t) \geq \eta \} = 0.
\end{equation}

In the rest of this section, we establish the above estimate and hence complete the proof of Theorem \ref{thm:main} under assumption ({\bf B}) and assumption ({\bf C}) respectively. For simplicity of notation, we suppress the dependence on $c\in \mathscr{C}^*$ in $\xi^c_\eps$ in the rest of the proof. %Under assumption ({\bf B}), we can apply Lemma \ref{Appv} which provides regularized $\hat{v}_\delta$ with bounded $L^\infty$
\medskip 

\subsection{Proof of the upper bound under assumption ({\bf B})}

For each fixed $\delta \in (0,1)$, we apply Lemma \ref{Appv} to find the random vector $\hat{v}_\delta$ satisfying \eqref{IeHd}. Note that $\hat{v}_\delta$ is obtained after mollification, and hence the associated scalar field $V_\delta(x,\omega)$, defined by \eqref{Mpot} with $v_\delta$ replaced by $\hat{v}_\delta$, is smooth and satisfies $\nabla V_\delta(x,\omega) = \hat{v}_\delta(\tau_x\omega)$. Moreover, \eqref{eq:vhatbdd} is satisfied by $\hat{v}_\delta$ and \eqref{ergoV} is satisfied by $V_\delta$. 

Let ${V}_{\delta,\epsilon}(x,\omega)= \epsilon V_{\delta}(\epsilon^{-1}x, \omega)$, then $\nabla V_{\delta,\epsilon}(x,\omega) = \hat{v}_{\delta}(\tau_{\epsilon^{-1}x}\omega)$, and for all $l>0$, almost surely we have
\begin{equation}\label{VE1}
    \lim_{\epsilon \to 0}\sup_{|x| \leq l} |V_{\delta,\epsilon}(x,\omega)| =0.
\end{equation}
Applying the rescaled It\^{o}'s formula \eqref{eq:rescaledIto}, we have
\begin{align}\label{E4D1}
     & V_{\delta,\epsilon}(X^{\epsilon}_{t},\omega) - V_{\delta,\epsilon}(x,\omega) \nonumber \\
    =& \int_{0}^{t} \langle \hat{v}_{\delta}(\tau_{\epsilon^{-1}X^{\epsilon}_{s}}\omega), c(s/\epsilon, X^{\epsilon}_s/\epsilon) \rangle \mathrm{d}s + \int_{0}^{t} \int_{\mathbb{R}^n} [V_{\delta,\epsilon}((X^{\epsilon}_{s^{-}}+y),\omega) - V_{\delta,\epsilon}(X^{\epsilon}_{s^{-}},\omega)]\widetilde{N}_\epsilon(\mathrm{d}s,  \,\mathrm{d}y)  \nonumber\\
    & + \epsilon^{\alpha - 1}\int_{0}^{t} \int_{\mathbb{R}^n}[V_{\delta,\epsilon}(X^{\epsilon}_{s^{-}}+  y,\omega) - V_{\delta,\epsilon}(X^{\epsilon}_{s^{-}},\omega)-\mathbf{1}_{B_1(0)}(y) \langle y,\hat{v}_{\delta}(\tau_{\epsilon^{-1}X^{\epsilon}_{s^-}}\omega)\rangle]\nu( \,\mathrm{d}y)\mathrm{d}s \nonumber \\
    = & \int_{0}^{ t} \langle \hat{v}_{\delta}(\tau_{\epsilon^{-1}X^{\epsilon}_{s}}\omega), c(s/\epsilon, X^{\epsilon}_{s}/\epsilon) \rangle \mathrm{d}s +   M^{\epsilon,\delta}_t + \epsilon^{\alpha - 1}\int_{0}^{t} \mathscr{L}V_{\delta,\epsilon}(X^{\epsilon}_{s},\omega)\mathrm{d}s,
\end{align}
where $\widetilde{N}_\eps$ is the compensated Poisson random measure of the rescaled L\'evy process $L^{\eps,\alpha}_\bullet = \eps L^\alpha_{\bullet/\eps}$, and in the third line we denote
\begin{equation}
\label{eq:Mbfsplit}
M^{\epsilon,\delta}_t
 :=\int_0^t\int_{\mathbb R^n}
\big[V_{\delta,\epsilon}(X^\epsilon_{s-}+y,\omega)
     -V_{\delta,\epsilon}(X^\epsilon_{s-},\omega)\big]
\widetilde N_\epsilon(\mathrm ds,\mathrm dy).
\end{equation}
We also have
\begin{align}
    \langle p, X^{\epsilon}_t - x \rangle = \int^{t}_{0} \langle p, c(s/\epsilon, X^{\epsilon}_{s}/\epsilon)\rangle \mathrm{d}s + \eps\langle p,  L^{\alpha}_{t/\eps} \rangle.
\end{align}
Note that $\eps^{\alpha-1}\mathscr{L}V_{\eps,\delta}(\cdot,\omega) = (\mathscr{L}V_\delta)(\cdot/\eps,\omega) = \mathscr{J}\hat{v}_\delta(\cdot/\eps,\omega)$. By \eqref{IeHd} of Lemma \ref{Appv}, or in the form \eqref{eq:vdeleq-1}, we have
\begin{equation*}
    \epsilon^{\alpha - 1}\mathscr{L}V_{\delta,\eps}(y,\omega) + H(p + \hat{v}_{\delta}(\tau_{\epsilon^{-1}y}\omega), \tau_{\epsilon^{-1}y}\omega) \leq \overline{H}(p) + \widetilde{\gamma}(\delta), \qquad y\in \R^n.
\end{equation*}
Note that $ \langle p, q\rangle \leq H(p,\omega) + L(q, \omega)$ for every $p,q \in\mathbb{R}^n$. It follows that
\begin{align}\label{E4D2}
    & \langle p, X^{\epsilon}_t - x \rangle + \int^{t}_{0} \langle \hat{v}_{\delta}(\tau_{\epsilon^{-1}X^{\epsilon}_{s}}\omega), c(s/\epsilon, X^{\epsilon}_{s}/\epsilon) \rangle \mathrm{d}s \nonumber \\
    = & \int^{t}_{0}\langle p+\hat{v}_{\delta}(\tau_{\epsilon^{-1}X^{\epsilon}_{s}}\omega), c(s/\epsilon, X^{\epsilon}_{s}/\epsilon)\rangle \mathrm{d}s + \epsilon\langle p,  L^{\alpha}_{t/\eps} \rangle \nonumber \\
    \leq & \int^{t}_{0} H(p+\hat{v}_{\delta}(\tau_{\epsilon^{-1}X^{\epsilon}_{s}}\omega),\tau_{X^{\epsilon}_{s}/\epsilon}\omega)\mathrm{d}s + \xi_{\epsilon}(t) + \epsilon\langle p,  L^{\alpha}_{t/\eps}\rangle \nonumber \\
    \leq & t\overline{H}(p)  - \epsilon^{\alpha -1}\int^{t}_{0}\mathscr{L}V_{\delta,\eps}(X^{\epsilon}_{s},\omega) \mathrm{d}s + \xi_{\epsilon}(t) + \epsilon\langle p,  L^{\alpha}_{t/\eps}\rangle.
\end{align}
Then together with (\ref{E4D1}), we obtain
\begin{align}\label{supersolC}
    \langle p, X^{\epsilon}_t - x \rangle - t(\widetilde{\gamma}(\delta)+\overline{H}(p))  - \xi_{\epsilon}(t)  \leq   V_{\delta,\epsilon}(x/\epsilon,\omega) -V_{\delta,\epsilon}(X^{\epsilon}_{t},\omega)  + \epsilon \langle p, L^{\alpha}_{t/\eps}\rangle + M_t^{\epsilon,\delta}.
\end{align}
Thus for every $\eta >0$,
\begin{align}\label{LE1}
    & Q^{c, \epsilon}_{x} \{ \langle p, (X^{\epsilon}_{t} -x) \rangle - t(\overline{H}(p)+\widetilde{\gamma}(\delta))- \xi_{\epsilon}(t) \geq \eta \} \nonumber \\
    \leq & Q^{c, \epsilon}_{x} \{|V_{\delta,\epsilon}(X^{\epsilon}_t,\omega) - V_{\delta,\epsilon}(x,\omega)| \geq \eta/2 \} + Q^{c, \epsilon}_{x}\left\{|M_t^{\epsilon,\delta}|+
    |\langle p,\eps L_{t/\eps}^{\alpha}\rangle| \geq \eta/2 \right\}.
\end{align}

To estimate the probability of the first set, since $|X^{\epsilon}_{t}|$ is uniformly integrable with respect to $\{Q^{c,\epsilon}_{x}: |x| \leq l, t\in [0,T], \epsilon\in [0,1], c\in \mathscr{C}^{\ast}  \}$, we use the locally uniform convergence \eqref{VE1}, and get
\begin{equation}
\label{LE2}
    \lim_{\epsilon\to 0}\sup_{c\in\mathscr{C}^{\ast}} \sup_{|x| \leq l} \sup_{0 \leq t \leq T} Q^{c, \epsilon}_{x}\{|V_{\delta,\epsilon}(X^{\epsilon}_t,\omega) - V_{\delta,\epsilon}(x,\omega)| \geq \eta \} = 0.
\end{equation}

For the probability of the second set, since $\mathbb{E}|\eps L^\alpha_{t/\eps}| \sim \eps^{1-\frac{1}{\alpha}}t^{\frac1\alpha}$ converges to zero (due to $\alpha >1$) uniformly for $t\in (0,T]$ and is independent of $(x,c)$, we only need to consider the contribution of $M^{\eps,\delta}_t$ defined in \eqref{eq:Mbfsplit}. We split $M^{\epsilon,\delta}$ as the sum of $M^{\epsilon,\delta,\flat}$ and $M^{\epsilon,\delta,\sharp}$ which, respectively, account for small-jumps in $\{|y|\leq1\}$ and big-jumps in $\{|y|>1\}$.

For the compensated small-jump martingale $M^{\eps,\delta,\flat}$, which has the expression
\begin{equation*}
    M^{\eps,\delta,\flat}_t = \int_0^t \int_{B_1(0)} V_{\delta,\eps}(X^\eps_{s^-}+y,\omega) - V_{\delta,\eps}(X^\eps_{s^-})\,\widetilde{N}_\eps(\dd s,\dd y),
\end{equation*}
we note the square integrability
\begin{equation*}
\begin{aligned}
    I^2_{t,x,\eps} &:= \int_0^t\int_{B_1(0)} \mathbb{E}|V_{\delta,\eps}(X^\eps_{s^-}+y,\omega) - V_{\delta,\eps}(X^\eps_{s^-})|^2 \eps^{\alpha-1}\nu(\dd y)\dd s \\
    &\le \|\hat{v}_\delta\|_{L^\infty(\Omega)}^2 \eps^{\alpha-1} t \int_{B_1(0)} |y|^2\nu(\dd y) < \infty,
    \end{aligned}
\end{equation*}
and, hence, can use It\^{o}'s isometry to get
\begin{align}\label{eq:MsmallB}
 \mathbb E^{Q_x^{c,\epsilon}}
 |M_t^{\epsilon,\delta,\flat}|^2 = I^2_{t,x,\eps} \le C_\delta \epsilon^{\alpha-1}t.
\end{align}

For $M^{\eps,\delta,\sharp}$, we write the
compensated integral as the Poisson integral minus its compensator and, using the fact that $|y|$ is $\nu(\dd y)$-integrable over $\{|y|>1\}$, and get 
\begin{equation}
\label{eq:MlargeB}
\begin{aligned}
 \mathbb E^{Q_x^{c,\epsilon}}
 |M_t^{\epsilon,\delta,\sharp}|
 %&\leq
 %2\epsilon^{\alpha-1}
 %\mathbf E^{Q_x^{c,\epsilon}}
 %\int_0^T\int_{|y|>1}
 %\left|V_{\delta,\epsilon}(z+y,\omega) -V_{\delta,\epsilon}(z,\omega)\right| 
 %\nu(\mathrm dy)\,\mathrm dr \nonumber\\
 \leq  2 \|\widehat v_\delta\|_{L^\infty(\Omega)} \epsilon^{\alpha-1} t
 \int_{|y|>1}|y|\nu(\mathrm dy)
 \leq C_{\delta}\epsilon^{\alpha-1}t.
\end{aligned}
\end{equation}
Note that the estimates \eqref{eq:MsmallB} and
\eqref{eq:MlargeB} are uniform with respect to $x$ and $c$. Applying Chebyshev's inequality, for every fixed $\delta$, we get
we get
\begin{equation}\label{LE3}
    \lim_{\epsilon\to 0}\sup_{c\in\mathscr{C}^{\ast}} \sup_{|x| \leq l} \sup_{0 \leq t \leq T} Q^{c, \epsilon}_{x}\left\{|M_t^{\epsilon,\delta}|+
    |\langle p,L_t^{\alpha,\epsilon}\rangle| \geq \eta \right\} = 0.
\end{equation}

Finally, combining \eqref{LE1} and \eqref{LE3} and sending $\delta\to 0$ (note then $\widetilde{\gamma}(\delta)\to 0$), we obtain
\begin{equation*}
    \lim_{\epsilon\to 0}\sup_{c\in\mathscr{C}^{\ast}} \sup_{|x| \leq l} \sup_{0 \leq t \leq T} Q^{c, \epsilon}_{x} \{ \langle p, (X^{\epsilon}_{t} -x) \rangle - t\overline{H}(p)- \xi_{\epsilon}(t) \geq \eta \} = 0.
\end{equation*}
This verifies \eqref{uppE} and completes the proof of Theorem \ref{thm:main} under Assumption ({\bf B}).

\medskip 

\subsection{The proof of upper bound under the assumption $(\bf{C})$}

Our goal is to prove \eqref{uppE}, for fixed positive numbers $l,T$ and $\eta$ and a fixed vector $p\in \R^n$, under assumption ({\bf C}). 

Given $c\in \mathscr{C}^*$, for $a$ and $b$ in $(1,\infty)$, for the SDE \eqref{eq:Xepsc} and the associated running cost $\xi_\eps$, define the stopping times
\begin{equation*}
\begin{aligned}
    &\tau_{\epsilon,a} := \inf \{t>0: \xi_{\epsilon}(t) \geq a \}, \quad \tau_{\epsilon,b} := \inf \{t>0: |X^{\epsilon}_{t} - x| \geq b\},\\
    &\tau_{\epsilon} = \tau_{\epsilon,a} \wedge \tau_{\epsilon,b}.
    \end{aligned}
\end{equation*}
By \eqref{A2L} and the estimate \eqref{sCset} satisfied for $c \in \mathscr{C}^*$ in assumption $(\bf A1)$, for $a$ sufficiently large,
\begin{equation}
    \lim_{\epsilon\to 0}\sup_{c\in\mathscr{C}^{\ast}} \sup_{|x| \leq l} \sup_{0 \leq t \leq T} Q^{c, \epsilon}_{x} \{\langle p, (X^{\epsilon}_{t} -x) \rangle - t\overline{H}(p)- \xi_{\epsilon}(t) \geq \eta, \tau_{\epsilon,a} < t\} = 0.
\end{equation}
Fix $a$ large so that the above is true. Estimating $X^\eps_t-x$ directly, we check that 
\begin{equation}
    \lim_{\epsilon\to 0}\sup_{c\in\mathscr{C}^{\ast}} \sup_{|x| \leq l} \sup_{0 \leq t \leq T} Q^{c, \epsilon}_{x} \{
\tau_{\eps,a}\ge t, \tau_{\epsilon,b} < t\}
\end{equation}
can be made arbitrarily small as $b \to \infty$. Hence, to prove \eqref{uppE} under Assumption ({\bf C}), it is sufficient to show that for each fixed $p\in\mathbb{R}^n$ and $l, T, \eta>0$, for $a$ and $b$ fixed sufficiently large, 
\begin{equation}
\label{43UP}
     \lim_{\epsilon\to 0}\sup_{c\in\mathscr{C}^{\ast}} \sup_{|x| \leq l} \sup_{0 \leq t \leq T} Q^{c, \epsilon}_{x} \{ \langle p, (X^{\epsilon}_{t} -x) \rangle - t\overline{H}(p)- \xi_{\epsilon}(t) \geq \eta, \tau_{\epsilon} \geq t \} = 0
\end{equation}

With the given $p\in \R^n$, we first apply Lemma \ref{Gardv} to obtain the random vector $v$ associated to $p$ which satisfies \eqref{eq:sublifted}. For each $\delta \in (0,1)$, let the random vector $v_\delta$ and the scalar field $V_\delta(x,\omega)$ be defined by \eqref{Mvcov} and by \eqref{Mpot} respectively. Finally, let $V_{\delta,\epsilon}(x,\omega)= \epsilon V_{\delta}(\epsilon^{-1}x, \omega)$. The following properties of $V_{\delta,\eps}$ will be helpful.

\begin{prop} For fixed $\eps,\delta \in (0,1)$ and $p\in \R^n$, let $V_{\delta,\eps}$ be defined as above. Under assumptions {\upshape ({\bf A})} and {\upshape ({\bf C})}, the following holds.
\begin{itemize}
\item[(1)] There exists some constant $C < \infty$ independent of $\eps,\delta$ such that 
\begin{equation}
\label{eq:lip-C}
|V_{\delta,\epsilon}(x,\omega)
-V_{\delta,\epsilon}(y,\omega)| 
\le C(\epsilon+|x-y|),
\end{equation}
\begin{equation}\label{eq:Vlinear-C}
    |V_{\delta,\epsilon}(x,\omega)|\leq C(\epsilon+|x|).
\end{equation}
\item[(2)] The following locally uniformly convergence holds almost surely in $\Omega$:
\begin{equation}
\label{eq:erg-delta}
    \lim_{\epsilon\to 0}\lim_{\delta\to 0}V_{\delta,\epsilon}(x,\omega) = \lim_{\epsilon\to 0}\lim_{\delta\to 0}\epsilon V_{\delta}(x/\epsilon,\omega)= 0. 
\end{equation}
\end{itemize}
\end{prop}
\begin{proof}By Lemma \ref{Ecv} and $\nabla V_\delta=\rho_\delta*\widetilde v$, we can find some constant $C<\infty$ depending on $p$ such that 
\[
 \sup_{\delta\in(0,1)}\sup_{z\in\mathbb R^n}
 \int_{B_1(z)}|\nabla V_\delta(y,\omega)|^\beta\,\mathrm dy
 \leq C.
\]
Since $\beta>n$, Morrey's inequality gives, uniformly in $\delta$,
\[
 |V_\delta(x,\omega)-V_\delta(y,\omega)|
 \leq C|x-y|^{1-n/\beta},
 \qquad |x-y|\leq1.
\]
Partitioning the segment from $x$ to $y$ into pieces of length at
most one and then rescaling yields
\begin{equation*}
 |V_\delta(x,\omega)-V_\delta(y,\omega)| \leq C(1+|x-y|).
\end{equation*}
Then
\begin{align*}
|V_{\delta,\epsilon}(x,\omega)
-V_{\delta,\epsilon}(y,\omega)|
=  \epsilon |V_\delta(x/\epsilon,\omega) -V_\delta(y/\epsilon,\omega)|
\leq  C\epsilon
\left(
1+\left|\frac{x-y}{\epsilon}\right|
\right) =
C(\epsilon+|x-y|).
\end{align*}
This is \eqref{eq:lip-C}. Since $V_{\delta,\epsilon}(0,\omega)=0$, we also have \eqref{eq:Vlinear-C}.

\medskip 

Moving on to \eqref{eq:erg-delta}, we first note that it is stronger than \eqref{VE1} because $\delta$ is sent to zero first. Again assumption ({\bf C}) is needed. Let $\widetilde V(x,\omega)$ be the vector field defined from $v$ via \eqref{eq:vtoV}. Then
\begin{equation*}
    \widetilde{V}_\delta(x,\omega) = (V(\cdot,\omega)*\rho_\delta)(x).
\end{equation*}
We see that $\nabla \widetilde{V}_\delta = \nabla V_\delta$, and hence
\begin{equation*}
    \widetilde{V}_\delta(x,\omega) = V_\delta(x,\omega) + \widetilde{V}_\delta(0,\omega).
\end{equation*}
For the random field $\widetilde{V}(x,\omega)$, its gradient is $\widetilde{v}(x,\omega)$ which is stationary and mean-zero, and $v \in L^\beta(\Omega)$ with $\beta > n$. By the Sobolev embedding theorem and the ergodic theorem (see, e.g., Theorem A.5 of \cite{MR2914944}), it is more or less well known that, $\mathbf{P}$-almost surely,
\begin{equation*}
    \lim_{\eps\to 0} \eps \widetilde{V}(x/\eps,\omega) = 0 \;\text{locally uniformly for $x \in \R^n$}.
\end{equation*}
Note that
\begin{equation*}
    \eps\widetilde{V}_\delta(x/\eps,\omega) = \int_{\R^n} \eps V(\frac{x}{\eps} - y,\omega)\rho_\delta(y)\,\dd y,
\end{equation*}
we also get the $\mathbf{P}$-a.s.,
\begin{equation*}
    \lim_{\epsilon\to 0}\lim_{\delta\to 0}\widetilde{V}_{\delta}(x/\eps,\omega) = 0, \qquad \;\text{locally uniformly for $x \in \R^n$}.
\end{equation*}
Since $\eps \widetilde{V}_\delta(0) \to 0$ as $\eps\to 0$, we get \eqref{eq:erg-delta}.
\end{proof}

Back to the proof of \eqref{43UP}, we first apply Lemma \ref{vVcompaison} with $\lambda = \ol{H}(p)$ and any fixed $c\in \mathscr{C}^*$. In particular, by rescaling \eqref{El44},   we have, almost surely with respect to $Q^{c, \epsilon}_{x}$,
\begin{align}\label{supersolE}
   &  \langle p, X^{\epsilon}_{t} - x \rangle - t\overline{H}(p) - \xi_{\epsilon}(t) \nonumber\\
    \leq & V_{\delta,\epsilon}(x,\omega) - V_{\delta,\epsilon}(X^{\epsilon}_{t},\omega) + r_c(\delta)t + \epsilon \langle p, L^{\alpha}_{t/\eps} \rangle  + M^{\eps,\delta}_t,
\end{align}
where where $M^{\eps,\delta}_t$ is defined in \eqref{eq:Mbfsplit} and $r_{c}(\delta)$ is a modulus of continuity in $\delta$ depending on the uniform norm of $c$. 
Applying It\^{o}'s formula to $V_{\delta,\epsilon}(X^{\epsilon}_{t},\omega)$, we get
\begin{align}\label{Vito}
    &V_{\delta,\epsilon}(X^{\epsilon}_{t},\omega) - V_{\delta,\epsilon}(x,\omega) \nonumber \\
     =& \int_{0}^{ t} \langle v_{\delta}(\tau_{\epsilon^{-1}X^{\epsilon}_{s}}\omega), c(s/\epsilon, X^{\epsilon}_{s}/\epsilon) \rangle \mathrm{d}s  + \epsilon^{\alpha-1}\int_{0}^{t} \mathscr{L} V_{\delta,\epsilon}(X^{\epsilon}_{s},\omega) \mathrm{d}s + M^{\eps,\delta}_{t}.  
\end{align}
After substituting (\ref{Vito}) into (\ref{supersolE}), we have, almost surely with respect to $Q^{c, \epsilon}_{x}$,
\begin{align}\label{supersolE1}
    &   \langle p, X^{\epsilon}_t - X^{\epsilon}_0 \rangle \nonumber - t\overline{H}(p)  - \xi_{\epsilon}(t)  \nonumber \\
    \leq & -\int_{0}^{ t} \langle v_{\delta}(\tau_{\epsilon^{-1}X^{\epsilon}_{s^-}}\omega), c(s/\epsilon, X^{\epsilon}_{s^-}/\epsilon) \rangle + \mathscr{J}\widetilde{v}_{\delta}(\epsilon^{-1}X^{\epsilon}_{s^-},\omega) \mathrm{d}s  + r_{c}(\delta)t+ \epsilon \langle p, L^{\alpha}_{t/\eps}\rangle.
\end{align}
We have seen before that $\epsilon \langle p, L^{\alpha}_{t/\eps}\rangle \to 0$ almost surely as $\epsilon \to 0$. Note also that $r_{c}(\delta)\to 0$ as $\delta \to 0$. 
For simplicity, in the remaining part of the proof we denote
\begin{equation*}
    F_{\delta,\epsilon}(s,X^{\epsilon}_{s^-},\omega) = \mathfrak{J}v_{\delta}(\tau_{\epsilon^{-1}X^{\epsilon}_{s^-}}\omega) + \langle c(\epsilon^{-1}s, \epsilon^{-1}X^{\epsilon}_{s^-}), v_{\delta}(\tau_{\epsilon^{-1}X^{\epsilon}_{s^-}}\omega)\rangle,
\end{equation*}
which is the integrand on the right hand side in \eqref{supersolE1}. As a result, in order to show (\ref{43UP}), it is enough to show that, $Q^{c, \epsilon}_{x}$-almost surely,
\begin{equation}\label{Q0}
    \lim_{\epsilon\to0}\sup_{c\in\mathscr C^*}\sup_{|x|\leq l} \sup_{0\leq t\leq T}\limsup_{\delta\to0}\mathbb E^{Q_x^{c,\epsilon}}\left(\int_0^{t\wedge\tau_\epsilon}F_{\delta,\epsilon}(s,X_{s^-}^\epsilon,\omega)\,\mathrm ds\right)^2=0.
\end{equation}

A key difference compared to the proof in the previous section is, we need to send $\delta\to 0$ first in the above equation. This is because the modulus of continuity $r_c(\delta)$ depends on $c$. In the previous section, under assumption ({\bf B}), the regularized $\hat{v}_\delta$ satisfies \eqref{IeHd} where the modulus of continuity $\widetilde{\gamma}(\delta)$ is independent of $x,t,c$, and $\delta$ is essentially fixed in the proof (sent to zero at last).

Writing the square of the integral as a double integral and viewing the resulted integrand as a conditional expectation with respect to $\mathcal{F}_s$, the filtration associated to $\{L^\alpha_\bullet\}$, we have
\begin{equation}
\label{eq:conditional-square}
\begin{aligned}
 &\mathbb E^{Q_x^{c,\epsilon}}
\left( \int_0^{t\wedge\tau_\epsilon}F_{\delta,\epsilon}(s,X_{s^{-}}^\epsilon,\omega)\mathrm ds \right)^2\\
=&2\mathbb E^{Q_x^{c,\epsilon}}
 \int_0^{t\wedge\tau_\epsilon}
 F_{\delta,\epsilon}(s,X_{s^-}^\epsilon,\omega) \mathbb E^{Q_x^{c,\epsilon}}\!\left[\int_s^{t\wedge \tau_\eps} F_{\delta,\epsilon}(r,X_{r^-}^\epsilon,\omega) \dd r\,\middle|\,\mathcal F_s\right]
 \dd s\\
 =&2\mathbb E^{Q_x^{c,\epsilon}}
 \int_0^{t\wedge\tau_\epsilon}
 F_{\delta,\epsilon}(s,X_{s^-}^\epsilon,\omega) 
\mathbb E^{Q_x^{c,\epsilon}}\!\left[V_{\delta,\epsilon}(X_{t\wedge\tau_\epsilon}^\epsilon,\omega)-V_{\delta,\epsilon}(X_s^\epsilon,\omega)\,\middle|\,\mathcal F_s \right]\dd s .
\end{aligned}
\end{equation}
where in the last step we applied \eqref{Vito} and removed the contribution of $M^{\eps,\delta}_{\bullet\wedge \tau_\eps}$ which is a martingale. 

Although the stopping time $\tau_{\epsilon,b}$ controls the pre-exit position $X^\epsilon_{\tau_{\epsilon,b}-}$, an unbounded exit jump can place $X^\epsilon_{\tau_{\epsilon,b}}$ outside every fixed ball.
We therefore need the overshoot estimate for \eqref{eq:conditional-square}. Since we fix $b$ sufficiently large, $b>l$ is assumed. Fix $R>2b$ and, for
$0\leq s\leq t\leq T$, let
\[
 G_{\epsilon,R}(s,t):=
 \left\{
 N_\epsilon((s,t]\times\{|y|>R\})=0
 \right\}.
\]
For any $r < (t\wedge \tau_\epsilon)$ we have $|X_r^\epsilon|\leq l+b$. Thus 
\[
G_{\epsilon,R}(s,t) \subset \big\{|X_{t\wedge\tau_\epsilon}^\epsilon|\leq l+b+R \big\}.
\]
Hence, on the event $G_{\eps,R}(s,t)$ we can bound $|V_{\delta,\epsilon}(X_{t\wedge\tau_\epsilon}^\epsilon,\omega)-V_{\delta,\epsilon}(X_s^\epsilon,\omega)|$ by twice of the maximum value of $|V_{\eps,\delta}(y)|$ in the ball $\{|y| \le R+b+l\}$.

On the complement set of $G_{\eps,R}(s,t)$, i.e., on $G^{\rm c}_{\eps,R}(s,t)$, a jump of size larger than $2b$ occurs for some $r\in (s,t]$. Since  $X^\eps_\bullet$ is in $\overline{B}_l(x)$ before hitting the stopping time $\tau_{\eps,b}$, we must have $\tau_{\eps,b}\le t$ and $X^\eps_{\tau_{\eps,b}^-} \in \overline{B}_l(x)$. Moreover, on this event we also have
\begin{equation*}
    |\Delta X_{t\wedge\tau_\epsilon}^\epsilon| 
 \leq  \sum_{s<r\leq t} |\Delta L_r^{\alpha.\eps}|\mathbf1_{\{| \Delta L_r^{\alpha,\eps}|>R\}} = \int_{|y|>R} |y| N_\eps((s,t],\dd y),
\end{equation*}
where $L^{\eps,\alpha}_\bullet$ denotes the rescaled process $\eps L^\alpha_{\bullet/\eps}$ and $N_\eps$ is the associated Poisson random measure. Due to \eqref{eq:lip-C}, this also leads to
\begin{equation}
    \label{E:jumppath}
    \begin{aligned}
    |V_{\delta,\epsilon}(X_{t\wedge\tau_\epsilon}^\epsilon,\omega)-V_{\delta,\epsilon}(X_s^\epsilon,\omega)| &\le |V_{\delta,\epsilon}(X_{(t\wedge\tau_\epsilon)^-}^\epsilon,\omega)-V_{\delta,\epsilon}(X_s^\epsilon,\omega)| + |\Delta V(X^\eps_{t\wedge \tau_\eps})|\\
    &\le 2\sup_{|y| \le b+l} |V_{\eps,\delta}(y)| + C\left(\eps + \int_{B_R^{\rm c}} |y| N_\eps((s,t],\dd y)\right).
    \end{aligned}
\end{equation}
Combining the above estimates, and using the fact that the compensator of $N_\eps$ is $\eps^{\alpha-1}\dd t\nu(\dd y)$, we get,
\begin{equation}\label{eq:overshoot-tail}
\begin{aligned}
 &\mathbb E^{Q_x^{c,\epsilon}}\!\left[\left|V_{\delta,\epsilon}(X_{t\wedge\tau_\epsilon}^\epsilon,\omega) - V_{\eps,\delta}(X^\eps_s)\right|
 \,\middle|\,\mathcal F_s
 \right]\\
\leq & 2\sup_{|y|\le R+b+l}|V_{\eps,\delta}(y)| + C\eps + C\mathbb{E}^{Q_x^{c,\epsilon}}\!\left[\int_{|y|>R} |y|\,N_\eps((s,t],dy) \,\Big| \, \mathcal{F}_s\right]\\
\leq & 2\sup_{|y|\le 2R}|V_{\eps,\delta}(y)| + C\eps + C\eps^{\alpha-1}(t-s)\int_{B_R^{\rm c}} |y|\nu(\dd y)\\
\le & 2\sup_{|y|\le 2R}|V_{\eps,\delta}(y)| + C_{n,\alpha,p}(\eps + \eps^{\alpha-1}(t-s)R^{1-\alpha}).
\end{aligned}
\end{equation}
Plugging this estimate into \eqref{eq:conditional-square}, we have
\begin{equation}
\label{eq:Clast}
\begin{aligned}
 & \mathbb E^{Q_x^{c,\epsilon}} \left( \int_0^{t\wedge\tau_\epsilon} F_{\delta,\epsilon}(s,X_s^\epsilon,\omega)\mathrm ds \right)^2 \nonumber \\
\leq & \left(4\sup_{|y|\le 2R}|V_{\eps,\delta}(y)| + C(\eps+T\eps^{\alpha-1}R^{1-\alpha}\right)\mathbb E^{Q_x^{c,\epsilon}}\int_0^{t\wedge \tau_\eps}|F_{\delta,\eps}(s,X^\eps_s,\omega)|\,\dd s\\
= & \left(4\sup_{|y|\le 2R}|V_{\eps,\delta}(y)| + C(\eps+T\eps^{\alpha-1}R^{1-\alpha}\right)\left[2\mathbb E^{Q_x^{c,\epsilon}}\int_0^{t\wedge \tau_\eps}(F_{\delta,\eps}(s))_+\,\dd s - \mathbb E^{Q_x^{c,\epsilon}} \int_0^{t\wedge \tau_\eps} F_{\delta,\eps}(s)\,\dd s\right].
\end{aligned}
\end{equation}
Here, $F_{\delta,\eps}(s)$ is a simplified notation for $F_{\delta,\eps}(s,X^\eps_s,\omega)$, and
\begin{equation*}
    (F_{\delta,\epsilon}(s,X^{\epsilon}_{s},\omega))_{+} = F_{\delta,\epsilon}(s,X^{\epsilon}_{s},\omega) \vee 0
\end{equation*}
denotes the positive part of $F_{\delta,\epsilon}$. Note that, using the martingale argument before (for $s=0$),
\begin{equation*}
\begin{aligned}
    \left|\mathbb E^{Q_x^{c,\epsilon}} \int_0^{t\wedge \tau_\eps} F_{\delta,\eps}(s)\,\dd s\right| &=  \left|\mathbb E^{Q_x^{c,\epsilon}}\left[V_{\eps,\delta}(X^\eps_{t\wedge \tau_\eps}) - V_{\eps,\delta}(x)\right]\right| \\
    &\le 2\sup_{|y|\le 2R}|V_{\eps,\delta}(y)| + C(\eps+T\eps^{\alpha-1}R^{1-\alpha}).
    \end{aligned}
\end{equation*}
In view of \eqref{eq:erg-delta}, for fixed $R>2b>2l$, the quantity on the right hand side above goes to zero as $\eps\to 0$, uniformly for $t\in [0,T]$, $c\in \mathscr{C}^*$, $x\in B_l(0)$ and $\delta\in (0,1)$. As a result, to verify \eqref{Q0}, it only remains to control
\begin{equation*}
    \mathbb{E}^{Q^{c, \epsilon}_{x}} \int^{t\wedge\tau_{\epsilon}}_{0} (F_{\delta,\epsilon}(s,X^{\epsilon}_{s},\omega))_{+}\mathrm{d}s = \mathbb{E}^{Q^{c, \epsilon}_{x}} \int^{t\wedge\tau_{\epsilon}}_{0} \big(\mathfrak{J}v_{\delta}(\tau_{\epsilon^{-1}X^{\epsilon}_{s}}\omega) + c(\epsilon^{-1}s, \epsilon^{-1}X^{\epsilon}_{s}), v_{\delta}(\tau_{\epsilon^{-1}X^{\epsilon}_{s}}\omega)\big)_{+}\mathrm{d}s.
    \end{equation*}
To this end, we use \eqref{eq:l54-main} to bound the integrand, and get
\begin{equation*}
\begin{aligned}
     &\mathbb{E}^{Q^{c, \epsilon}_{x}} \int^{t\wedge\tau_{\epsilon}}_{0} (F_{\delta,\epsilon}(s,X^{\epsilon}_{s},\omega))_{+}\mathrm{d}s \\
     &\qquad\qquad \le \mathbb{E}^{Q^{c, \epsilon}_{x}} \int^{t\wedge\tau_{\epsilon}}_{0} \Big(\ol{H}(p) + L(c(s/\eps,X_{s/\eps}),\tau_{X_{s/\eps}}\omega) - \langle c(s/\eps,X_{s/\eps}), p\rangle + r_c(\delta)\Big)_+\,\dd s
\end{aligned}
\end{equation*}
Since $c\in \mathscr{C}^*$, using the estimates in Lemma \ref{Cset}, we easily find a bound
\begin{equation*}
    \sup_{c\in \mathscr{C}^*} \sup_{|x| \le l} \sup_{\delta\in (0,1)} \mathbb{E}^{Q^{c, \epsilon}_{x}} \int_0^{t\wedge \tau_\eps} \big(F_{\delta,\eps}(s,X^\eps_s,\omega)\big)_+\,\dd s \le C(T,p,H,\alpha).
\end{equation*}
Hence we establish \eqref{Q0} and complete the proof under assumption ({\bf C}). 

%%%%%%%%
%%%%%%%%

\section*{Acknowledgments}
The work of WJ is partially supported by the NSFC Grant No.\,12571220 and by the New Cornerstone Investigator Program 100001127.

%%%%%%%%
%\appendix

%%%%%%%%
\bibliographystyle{abbrv}
\bibliography{Ref}

\end{document}